\documentclass[a4paper]{article}
\usepackage{amsmath,amssymb,amsthm,graphicx,rotating}
\newcounter{mycount}
\newenvironment{romlist}{\begin{list}{(\roman{mycount})}{\usecounter{mycount}\labelwidth=1cm\itemsep 0pt}}{\end{list}}
\newenvironment{Romlist}{\begin{list}{(\Roman{mycount})}{\usecounter{mycount}\labelwidth=1cm\itemsep 0pt}}{\end{list}}

\numberwithin{equation}{subsection}
\newtheorem{theorem}[equation]{Theorem}

\newtheorem{proposition}[equation]{Proposition}
\newtheorem{corollary}[equation]{Corollary}
\newtheorem{lemma}[equation]{Lemma}
\newtheorem{definition}[equation]{Definition}
\newtheorem{remark}[equation]{Remark}
\DeclareMathSymbol{\leqslant}{\mathalpha}{AMSa}{"36}
\DeclareMathSymbol{\geqslant}{\mathalpha}{AMSa}{"3E}
\renewcommand{\le}{\;\leqslant\;}
\renewcommand{\ge}{\;\geqslant\;}
\def\AA{{\mathbb A}}\def\BB{{\mathbb B}}\def\DD{{\mathbb D}}\def\EE{{\mathbb E}}\def\HH{{\mathbb H}}\def\MM{{\mathbb M}}\def\NN{{\mathbb N}}\def\PP{{\mathbb P}}\def\QQ{{\mathbb Q}}\def\RR{{\mathbb R}}\def\WW{{\mathbb W}}\def\ZZ{{\mathbb Z}}
\def\sA{\mathcal A}\def\sB{\mathcal B}\def\sC{\mathcal C}\def\sD{\mathcal D}\def\sE{\mathcal E}\def\sF{\mathcal F}\def\sH{\mathcal H}\def\sJ{\mathcal J}\def\sL{\mathcal L}\def\sN{\mathcal N}\def\sP{\mathcal P}\def\sR{\mathcal R}\def\sS{\mathcal S}\def\sU{\mathcal U}\def\sV{\mathcal V}\def\sW{\mathcal W}

\def\O{\mathrm{O}}
\def\qq{\qquad}
\def\q{\quad}
\def\oo{\infty}

\def\lest{\le_{\!\!\mathrm{st}}\ }
\def\gest{\ge_{\!\!\mathrm{st}}\ }
\def\a{\alpha}
\def\b{\beta}
\def\d{\delta}
\def\eps{\varepsilon}
\def\l{\lambda}
\def\s{\sigma}
\def\t{\tau}
\def\th{\theta}
\def\n{\eta}
\def\Si{\Sigma}
\def\L{\Lambda}
\def\Om{\Omega}
\def\om{\omega}
\def\De{\Delta}

\def\ba{\mathbf a}
\def\bb{\mathbf b}
\def\be{\mathbf e}
\def\bn{\mathbf n}
\def\bu{\mathbf u}
\def\bx{\mathbf x}

\def\sm{\setminus}

\def\lra{\leftrightarrow}
\def\nlra{\nleftrightarrow}

\def\emptyset{\varnothing}
\def\bcdot{\,\cdot\,}
\def\qtext#1{\q\text{#1}\q}

\def\pc{p_\mathrm{c}}

\def\pd{\partial}
\def\dd{\,\mathrm{d}}

\def\fv{\vec{f}}

\def\cg{c_\mathrm{cg}}
\def\cgf{c_\mathrm{cg}'}
\def\cdil{C_\mathrm{dil}}

\def\c{\mathrm{c}}
\def\f{\mathrm{f}}
\def\g{\mathfrak{g}}
\def\w{\mathrm{w}}
\def\bc{\b_\c}
\def\hbc{\hat{\b}_\c}
\def\pc{p_\c}
\def\hm{\hat{\mu}^{J,-,h}}
\def\disconnected{\mathrm{D}}
\def\thmconstant{C_0}
\def\propconstanta{C_1}
\def\truncationConstant{C_\mathrm{cut}}
\def\truncationConstantB{c_\mathrm{cut}}
\def\energyStabilityConstant{c_\mathrm{stb}}
\def\stabilityConstant{c_\mathrm{hs}}

\def\BV{\mathrm{BV}}
\def\Bcritical{B_\c^\th}
\def\Bcriticall{B_\c^{2\pi}}
\def\Bmax{B_\mathrm{max}^\th}
\def\Bmaxx{B_\mathrm{max}^{2\pi}}
\def\Bmin{B_\mathrm{min}^\th}
\def\Bminn{B_\mathrm{min}^{2\pi}}
\def\epsWMbc{\eps_\mathrm{wm}}% (w,-) boundary conditions, h=0
\def\epsPMbc{\eps_\mathrm{pm}}% (+,-) boundary conditions, h=0
\def\epsStab{\eps_\mathrm{stb}}% Stability of minimum energy profiles, h>0
\def\epsCG{\eps_\mathrm{cg}}% Coarse graining
\def\epsC{\eps_\sC}% Definition of event $\sC(\epsC)$
\def\epsHS{\eps_\mathrm{hs}}% Hausdorff stability + HS implies spatial mixing
\def\epsSG{\eps}% Spectral gap
\def\epsBlock{\eps_\mathrm{block}}% Spectral gap/block dynamics

\title{Metastability in the dilute Ising model}
\author{
Thierry Bodineau\\
{\small \'Ecole Normale Sup\'erieure}\\
{\small \tt bodineau@dma.ens.fr}\\
\and
Benjamin Graham\\
{\small University of Warwick}\\
{\small \tt b.graham@warwick.ac.uk}\\
\and
Marc Wouts\\
{\small Universit\'e Paris 13}\\
{\small \tt wouts@math.univ-paris13.fr}\\
}

\begin{document}
\maketitle

\begin{abstract}
Consider Glauber dynamics for the Ising model on the hypercubic lattice with a positive magnetic field. Starting from the minus configuration, the system initially settles into a metastable state with negative magnetization. Slowly the system relaxes to a stable state with positive magnetization. Schonmann and Shlosman showed that in the two dimensional case the relaxation time is a simple function of the energy required to create a critical Wulff droplet.

The dilute Ising model is obtained from the regular Ising model by deleting a fraction of the edges of the underlying graph. In this paper we show that even an arbitrarily small dilution can dramatically reduce the relaxation time.
This is because of a catalyst effect---rare regions of high dilution speed up the transition from minus phase to plus phase.

\end{abstract}

\tableofcontents

\section{Introduction}
\subsection{Metastability in the Ising model}

Consider Glauber dynamics for the supercritical Ising model on the hypercubic lattice ($d\ge 2$) started in the minus configuration but with a positive external magnetic field $h$. Aizenman and Lebowitz predicted that the model initially settles in a metastable minus phase, eventually relaxing to the plus phase on a time scale that grows exponentially with $1/h^{d-1}$ \cite{AizenmanLebowitzMetastabilityBootstrap}.

To be more precise, let $\b$ denote the inverse-temperature and let $\bc$ denote the critical inverse-temperature. Suppose $\b>\bc$. Let $\mu^+,\mu^-$ denote the plus and minus phases of the equilibrium Ising model. Start the Glauber dynamics at time $0$ with all vertices initially taking minus spin. Let $\s^{0,-}_{t}$ denote the state of the Glauber dynamics at time $t$. With $\b$ fixed, let $h \to 0$ with $t=\exp(\l/h^{d-1})$. A heuristic argument suggests that if $\l_1$ is sufficiently small and $\l_2$ is sufficiently large then for every local observable $f$:
\begin{romlist}
\item $\EE[f(\s^{0,-}_{t})]\to \mu^-(f)$ if $\l < \l_1$.
\item $\EE[f(\s^{0,-}_{t})]\to \mu^+(f)$ if $\l > \l_2$.
\end{romlist}
Part (i) is a lower bound on the relaxation time and part (ii) is an upper bound.  Schonmann proved this behavior in dimensions $d\ge 2$ \cite{SchonmannSlowDropletRelaxation}.  However, his proof left open the question of whether or not $\l_1=\l_2$.  Schonmann and Shlosman settled this question in dimension two, proving that the above holds with $\l_1=\l_2$; the transition is sharp in a logarithmic sense \cite{schonmann-shlosman}.  Their proof refines the heuristic argument and shows that the critical value of $\l$ is a simple function of the surface tension of the Wulff shape. The proof takes advantage of specific features of the two dimensional Ising model such as duality. When considering disordered models, in two and higher dimensions, these simplifying features no longer exist. New arguments from the $L_1$-theory of phase coexistence have to be used instead.

The focus of this paper will be the dilute Ising model. For the purpose of comparison, we note that the proof of our main result (Theorem~\ref{theorem:main} below) implies that the upper bound of \cite{schonmann-shlosman} extends to higher dimensions.
We believe that our method of proof is valid for all $\b>\bc$ but we did not make verifying this a priority. To avoid certain technicalities we assume that $\b>\b_0$ and $\b\not\in\sN$ (see Section \ref{dilute-ising-model}).

Let $\mu^h$ denote the equilibrium, undiluted Ising measure with a magnetic field $h>0$. 
With reference to \eqref{A-def}, the cost of creating a critical droplet under $\mu^h$ is $\mathsf{E}^{2\pi}_\c/h^{d-1}$ and $\mathsf{E}^{2\pi}_\c=\O(\b)$. 
Define a {\em local observable} to be a function that only depends on the spins in the region $[-1/h,1/h]^d$.

\begin{theorem}\label{theorem:p1}
Consider the value
\[
\l_2=\frac{\mathsf{E}^{2\pi}_\c}{d+1}.
\]
Let $\l>\l_2$.  For any positive number $\thmconstant$ there is a constant $C>0$ such that for every local observable $f$ and any $h>0$,
\begin{align*}
  \left|
\EE \left(
   f\left(\s^{0,-}_{\exp(\l/h^{d-1})}\right) \right) - \mu^h (f)
  \right|
  \le C \|f\|_\oo \exp(-\thmconstant/h).
\end{align*}
\end{theorem}

This is an improvement on the upper bound in \cite{SchonmannSlowDropletRelaxation} and corresponds to the upper bound predicted by the heuristic of \cite{schonmann-shlosman}.
Proving rigorously the lower bound suggested by \cite{schonmann-shlosman} in dimensions three and higher requires the development of new arguments which we postpone to a future work.

The dilute Ising model is a variant of the Ising model that is obtained by randomizing the Ising model edge coupling strengths. The impact of dilution on the relaxation of Glauber dynamics has been studied in \cite{CMM1, Ma}.
In the Griffiths phase, which corresponds to the sub-critical regime, the disorder is proven to lead to a slowdown of the dynamics.  In the phase transition regime, the metastability has been investigated for the random field Curie-Weiss model \cite{BianchiBovierIoffe}.

We will consider the Ising model on $\ZZ^d$ diluted in the simplest way possible.  Independently, delete each edge with probability $1-p$.  When $p$ is sufficiently large, the remaining edges form a supercritical percolation cluster.  From this point of view, the Ising model is a special case of the dilute Ising model corresponding to $p=1$.
It is natural to ask how the relaxation time depends on $p$.  In this paper we show that even a small dilution can greatly reduce the relaxation time.

\subsection{The dilute Ising model}\label{dilute-ising-model}

Let $\ZZ^d$ represent the hypercubic lattice.  The Ising model assigns each site of $\ZZ^d$ a spin of $\pm1$. Let $\Si=\{\pm1\}^{\ZZ^d}$ denote the set of Ising configurations.

Let $E=\{\{x,y\}:\|x-y\|_1=1\}$ denote the set of nearest neighbor edges of $\ZZ^d$. The equilibrium Ising measure with local coupling strengths $J=(J(e):e\in E)$ and external magnetic field $h$ is defined using the formal Hamiltonian
\begin{align}\label{formal hamiltonian}
-\frac12 \sum_{e = \{x,y\} \in E} J(e) \s(x) \s(y) - \frac12 \sum_{x\in\ZZ^d} h \s(x) ,\qq \s \in \Si.
\end{align}
We will consider local coupling strengths with the Bernoulli distribution.  Let $\QQ$ denote the product measure such that for each edge $e$, $\QQ(J(e)=1)=p$ and $\QQ(J(e)=0)=1-p$.

It is well known that when $h\not=0$, the Ising measure $\mu^{J,h}$ is well defined by the Gibbs formalism for any inverse-temperature $\b> 0$ and local coupling strengths $J\ge0$.  Consider the spontaneous magnetization of the Ising measure,
\begin{align}\label{m star}
m^* = \lim_{h\to 0+} \QQ \left[\mu^{J,h} \left( \s(0) \right)\right].
\end{align}
When $m^* > 0$ there is said to be phase coexistence.  For such $\b$ there are two different Gibbs measures at $h=0$, corresponding to the limits $h\to 0+$ and $h\to 0-$.

It is shown in \cite{ChayesChayesFrohlich} that if the $J$-positive edges percolate then there is phase coexistence in the dilute Ising model at low temperatures.  In our settings, this means that the critical inverse-temperature
\begin{align*}
\bc = \inf \left\{ \b > 0 : m^* > 0 \right\}
\end{align*}
is finite if and only if $p > \pc$, where $\pc$ is the threshold for bond percolation on $(\mathbb{Z}^d,E)$.

As well as defining the equilibrium Ising model, the formal Hamiltonian defines a dynamic model.  Let $(\s^{0,-}_t)_{t\ge 0}$ denote a Markov chain on the set $\Si$ of Ising configurations, starting at time $0$ with minus spins everywhere, and evolving with time according to Glauber dynamics.  Given a set of coupling strengths $J$, let $\EE_J$ denote expectation with respect to the Glauber dynamics.  Our results extend to some other dynamics such as the Metropolis dynamics (see Section~\ref{sec:heat bath}).

A quantity denoted $\cdil$ is defined in Section~\ref{sec:big-proof} that satisfies $\cdil=\O(\log\frac{1}{1-p})$ as $p\to 1$.  For the rest of the paper consider $p$ to be fixed in the range $(\pc,1)$.

Let $\b_0$ denote the minimum value such that for all $\b>\b_0$ the assumptions of slab percolation (see Section~\ref{sec:cg}) and spatial mixing (see Section~\ref{sec:mixing}) hold.  Let $\sN\subset(0,\oo)$ denote the set of zero measure defined by \eqref{sN}.  For the rest of the paper the inverse-temperature $\b$ should be assumed to be greater than $\b_0$ and not in $\sN$.

For $\th\in(0,\pi)$ let $\mathsf{E}^\th_\c$ denote the cost, up to a factor of $h^{d-1}$, of creating a critical plus droplet in a cone with angle $\th$. $\mathsf{E}^\th_\c=\O(\b\th^{d-1})$ and is defined in Section~\ref{sec:critical droplets}.

\begin{theorem}\label{theorem:main}
For $\th\in(0,\pi)$ consider the value
\[
\l_2^\th=\frac{\mathsf{E}^\th_\c + \cdil \th^{-1}}{d+1}.
\]
\begin{romlist}
\item
Let $\l>\l_2^\th$.  For any positive number $\thmconstant$, there are constants $C,c>0$ such that for any $h>0$, for every local observable $f$,
\begin{align*}
\QQ \left[\, \left| \EE_J \left( f\left(\s^{0,-}_{\exp(\l/h^{d-1})}\right) \right) - \mu^{J,h} (f) \right| \le C \|f\|_\oo \exp(-\thmconstant/h) \, \right]& \\ \ge 1-C\exp(-c /\sqrt{h})&.
\end{align*}
\item
However close $p$ is to one, at low temperatures the diluted Ising model relaxes much more quickly than the corresponding undiluted Ising model; with reference to Theorem~\ref{theorem:p1}, as $\b\to\oo$,
\[
\frac1{\l_2} \inf_{\th\in(0,\pi)} \l_2^\th \to 0.
\]
\end{romlist}
\end{theorem}

We expect, based on the undiluted Ising model \cite{BodineauIsing}, that the slab percolation threshold is equal to $\bc$. Further study of slab percolation and spatial-mixing properties for the dilute Ising model would likely extend the domain of validity of Theorem~\ref{theorem:main} down to the critical point.

\subsection{Heuristic}\label{sec:heuristic}

The metastability phenomenon for the undiluted Ising model \cite{schonmann-shlosman} is related to the rate of nucleation of plus droplets with linear size order $1/h$.  Consider a small neighborhood of the origin.  Initially all the spins are minuses.  Small clusters of plus spins quickly form and then disappear.  After a short time the system looks like it has reached equilibrium with minus spins in the majority.  However, if we look at a much larger region we will be in for a surprise.  A small number of larger droplets of plus spin will have formed and started to spread.  They will eventually merge and cover the whole region, leaving the majority of spins in the plus state.

The rate at which droplets of plus phase form, and what happens to the droplets once they have formed, depends on their energy.  Let $\sV\subset\RR^d$ with unit volume.  For $b>0$ let $\mathsf{E}^\sV(b)$ denote, up to a factor of $h^{d-1}$, the energy of a plus droplet with the shape $(b/h)\sV$.  $\mathsf{E}^\sV(b)$ can be estimated as a balance between the surface tension at the phase boundary and the effect of the magnetic field $h$,
\begin{equation}
\label{eq: energy}
\mathsf{E}^\sV(b)/h^{d-1} = (b/h)^{d-1} \sF(\sV) - h m^* (b/h)^d.
\end{equation}
Here $\sF(\sV)$ is the surface tension of $\sV$ (see Section~\ref{sec:st}) and $m^*$ is the mean magnetization in the plus phase \eqref{m star}.

Let $B_\c^\sV = \frac{(d-1) \sF(\sV)}{d m^*}$.  The energy function $\mathsf{E}^\sV(b)$ is increasing on the interval $(0,B_\c^\sV)$ and decreasing beyond $B_\c^\sV$.  Droplet with $b<B_\c^\sV$ are unstable and tend to be eroded by the surrounding minus spins.  Droplets with $b>B_\c^\sV$ are expected to spread.  The nucleation of a droplet with $b>B_\c^\sV$ requires that the system overcomes an energy barrier $\mathsf{E}^\sV_\c/h^{d-1}$, where $\mathsf{E}^\sV_\c := \mathsf{E}^\sV(B_\c^\sV)$.  Given the inverse-temperature $\b$ there is a unique shape $\sW$ known as the Wulff shape with minimal surface tension; see the definition of $\sW_{2\pi}$ in Section~\ref{sec:wulff and summertop shapes}. Setting $\sV=\sW$ minimizes $\mathsf{E}^\sV_\c$. The critical droplet shape is $(B_\c^\sW/h) \sW$.

In any small neighborhood the rate at which copies of the critical droplet form is approximately $\exp(-\mathsf{E}^\sW_\c/h^{d-1})$.  Droplets larger than the critical droplet spread out with roughly uniform speed and eventually invade the whole space.  The space-time cone of points from which one can reach the origin by time $t$ (when growing at a fixed speed) has size $\O(t^{d+1})$.  If $t=\exp(\l/h^{d-1})$ with
\begin{align}\label{eq:lambda_c}
\l > \l_c = \frac{\mathsf{E}^\sW_\c}{d+1}
\end{align}
then we should expect to see a critical droplet form, and then spread to cover the origin, by time $t$.  This heuristic picture has been turned into a rigorous proof for the two dimensional Ising model \cite{schonmann-shlosman}.

The dilute Ising model is self averaging so the quenched magnetization and the quenched surface tension can unambiguously be defined almost surely with respect to the dilution measure $\QQ$.
It is tempting to try to adapt the previous heuristic to the case of the dilute Ising model using the quenched surface tension in \eqref{eq: energy} to describe the typical cost of phase coexistence.
However, we must be careful. In much simpler models, such as random walk in random environment, it is well known that a small amount of randomness can change the asymptotic behavior.

Dilution seems to be capable of slowing down the dynamics.  Consider an expanding droplet of plus phase. If it encounters an area of high dilution it may get blocked and have to seep around the obstruction, slowing down its progress.

However, dilution can also speed up the dynamics. The limiting factor in the undiluted Ising model is the rate at which plus droplets nucleate. Nucleation of plus droplets is infrequent due to the high cost of phase coexistence on their boundaries. The dilution creates atypical regions, which we will call {\em catalysts}, where the surface tension is unusually low and so the rate of nucleation is unusually high.

The natural human response to catalysts is to try and classify them. Some catalyst do not seem to have much effect on the relaxation time. Consider (when $d=2$) a circle where all the edges crossing its perimeter have been diluted. If a plus droplet forms inside the circle, there is no way for it to spread outwards.

We therefore want to focus on catalysts that create a sheltered region to help plus droplets nucleate, but are not so closed off they prevent plus droplets from escaping. There seem to be two competing factors. Large catalysts will be relatively rare and so the droplets they help to nucleate will take a long time to reach the origin. Conversely, small catalysts cannot do a great deal to increase the rate at which critical droplets nucleate.

We conjecture that there is an optimal catalyst shape that determines the relaxation time of the system. However, we do not know how to calculate the optimal shape.  In this paper we look at a restricted class of catalysts: surfaces of diluted edges that form open-bottom cones. We control the nucleation rate in the cones, and the subsequent growth of the droplet to regions of more typical dilution. This approach leads to an upper bound on the relaxation time that is much smaller than the time predicted by the formula \eqref{eq:lambda_c} with quenched surface tension. Indeed, part (ii) of Theorem~\ref{theorem:main} shows that asymptotically in $\b$ the values of $\l_2$ differ greatly.

We do not address the issue of the lower bound for the metastable time for disordered models. We believe that the more important point is to show the existence of the catalyst effect of the disorder.

\subsection{Outline of the paper}
In Section~\ref{sec:prop} we define the dilute Ising model and recall some of its basic features. The random-cluster representation is used to state a coarse graining property.

In Sections~\ref{sec:L1} and \ref{sec:wulff} we look at the Ising model without a magnetic field.
In Section~\ref{sec:L1} we describe the $L_1$-theory of phase coexistence. The theory can describe both the typical cost of phase coexistence and the cost of phase coexistence in the neighborhood of catalysts. To combine the two cases we consider the cone $\sA_\th:=\{\bx\in\RR^d: x_1 \ge \|\bx\|_2 \cos (\th/2)\}$ where either $\th\in(0,\pi)$ or $\th=2\pi$.
In Section~\ref{sec:wulff} we look at generalizations of the Wulff shape to $\sA_\th$. The Wulff shape is the shape with minimal surface tension given its volume. The Wulff shape can be used to quantify the large deviations of the equilibrium Ising model.

In Section~\ref{sec:technical} we reintroduce the magnetic field. We justify the energy function featured in the heuristic. We prove regularity results concerning cluster boundaries. The motivation for this is to study the spectral gap of the dilute Ising model in finite regions with various boundary conditions.

Finally in Section~\ref{sec:big-proof} we use the accumulated results to prove Theorem~\ref{theorem:main}. We do this by proving that the cone shaped regions act as catalysts. To show that the clusters of plus phase formed in the catalysts grow we consider another type of cone: space-time cones that are Wulff shaped spatially and growing in size with time.

\subsection{Notation}

Throughout the paper $C,c,c_\mathrm{stb},c_\mathrm{hs}$, etc, will be used to refer to positive numbers that may depend on $p,\b$ and $\th$ but not on $h$.  We will recycle $C$ and $c$ to refer to various less important positive constants; the values they represent will change from appearance to appearance.

Let $\sS^{d-1}$ denote the set of unit vectors in $\RR^d$. Let $\be_1,\dots,\be_d$ denote the canonical basis vectors. We will use bold to differentiate continuous variables $\bx\in\RR^d$ from lattice points $x\in\ZZ^d$.

Consider $\sA,\sB \subset \RR^d$, $\bx\in\RR^d$ and $c\ge0$.  Let $\sA+\sB=\{\ba+\bb:\bb\in\sA,\bb\in\sB\}$ denote the sum of the two sets.  Let $\sA+\bx$ denote the translation of $\sA$ by $\bx$.  Let $c\sA$ denote $\{c\,\ba:\ba\in \sA\}$, the set $\sA$ scaled by a factor of $c$.

\section{Properties of the dilute Ising model}\label{sec:prop}

\subsection{Definition of $\mu^{J,\zeta,h}_\L$}

Let $J=(J(e):e\in E)$ be a given realization of the coupling strengths.  We will now define formally the Ising measure $\mu^{J,\zeta,h}_\L$ with a magnetic field $h\in\RR$ and boundary conditions $\zeta\in\Si$ on a finite domain $\L\subset\ZZ^d$ at inverse-temperature $\b> 0$.

Define the {\em external vertex boundary} $\pd \L$ of $\L$:
\begin{align*}
\pd\L&=\pd^+\L\cup \pd^-\L \qq\text{where}\\ \pd^\pm \L&= \{x\not\in\L: \exists y\in\L,\
\{x,y\}\in E \text{ and }
\zeta(x)=\pm1\}.
\end{align*}
Taking w to stand for {\em wired}, define edge sets for $\L$:
\begin{eqnarray*}
E (\L) & = & \left\{ \left\{ x, y \right\}\in E : x, y \in \L\right\},\\ E^\pm(\L) & = & \left\{ \left\{ x, y \right\}\in E : x \in \L \text{ and } y\in\pd^\pm \L\right\},\\ E^\w (\L) & = & E (\L) \cup E^\pm(\L).
\end{eqnarray*}
The set of spin configurations compatible with $\zeta$ outside $\L$ is
\[
\Si_\L^\zeta:=\{\s\in\Si:\forall x \not\in \L,\ \s(x)=\zeta(x)\}  .
\]
Changing a Hamiltonian by an additive constant does not change the resulting measure; with reference to \eqref{formal hamiltonian} the Ising Hamiltonian $H^{J,\zeta,h}_\L:\Si_\L^\zeta\to\RR$ can be defined by
\begin{align*}
H^{J,\zeta,h}_\L (\s) = \sum_{e =\{x, y\} \in E^\w (\L)} J(e) 1_{\{\s(x)\not=\s(y)\}}+\sum_{x\in \L} h 1_{\{\s(x)=-1\}}.
\end{align*}
The dilute Ising measure $\mu^{J,\zeta,h}_{\L}$ at inverse-temperature $\b$ is defined by
\begin{align*}
  \mu^{J,\zeta,h}_{\L} (\{\s\}) = \frac{1}{Z^{J,\zeta,h}_\L} \exp \left( - \b H_{\L}^{J,\zeta,h} (\s) \right)
\end{align*}
where $Z^{J,\zeta,h}_\L$ is a normalizing constant, the partition function, defined by
\begin{align}\label{partition function}
Z^{J,\zeta,h}_\L = \sum_{\s \in \Si^\zeta_\L} \exp \left( - \b H_{\L}^{J,\zeta,h} (\s) \right).
\end{align}
We have used $\s$ above to index summations over $\Si^\zeta_\L$.  It has also been used as a random variable---the mean spin at the origin is written $\mu^{J,\zeta,h}_\L(\s(0))$.  Furthermore, given a set $V\subset \ZZ^d$ we will write $\s(V)$ to denote the average spin in $V$,
\begin{align}\label{average spin}
\s(V)=\frac{1}{|V|}\sum_{x\in V} \s(x)\in[-1,+1].
\end{align}

\subsection{The random-cluster representation for $\mu^{J,\zeta,h}_\L$}

The spin-spin correlations in the Ising model can be described by the $q=2$ case of the random-cluster model \cite{G-RC}. We have to be extra careful because of the general boundary conditions $\zeta\in\Si$, dilute coupling strengths $(J(e))$, and the magnetic field $h\ge 0$. In this section we will describe a random-cluster representation $\phi^{J,\zeta,h}_\L$ for the Ising model $\mu^{J,\zeta,h}_\L$ and a joint measure $\varphi^{J,\zeta,h}_\L$.

The Ising measure $\mu^{J,\zeta,h}_\L$ was defined using the graph $(\L,E^\w(\L))$. Add to this graph a ghost vertex $\g$ through which the magnetic field will act, and a set of ghost edges $E^\g(\L)=\{\{\g,x\},x\in\L\}$.

When defining the random-cluster model on a given graph, for each edge $e$ there is an interaction-strength parameter $p_e\in[0,1]$. There is also another parameter, $q>0$, that influences the number of clusters that are formed. In order to describe the correlations of the dilute Ising model we will fix
\begin{align*}
q=2 \qtext{and} p_e=\begin{cases}
1-\exp(-\b J(e)),& e\in E^\w(\L),\\1-\exp(-\b h), & e\in E^\g(\L).
\end{cases}
\end{align*}
The state space of $\phi^{J,\zeta,h}_\L$ is $\Om_\L=\{0,1\}^{E^\w (\L)\cup E^\g(\L)}$. With $\om\in\Om_\L$, an edge $e$ is {\em open} if $\om(e)=1$ and {\em closed} if $\om(e)=0$.  Two vertices of $\L$ are {\em connected} if they are joined by paths of open edges either
\begin{romlist}
\item to each other, or
\item both to $\pd^+\L\cup\{\g\}$, or
\item both to $\pd^-\L$.
\end{romlist}
A cluster is a maximal collection of connected vertices.
Let $V_+$, $V_-$ denote the clusters connected to $\pd^+\L\cup\{\g\}$, $\pd^-\L$, respectively.
Let $n=n(\L,\om)$ count the number of other clusters in $\L$.
Label these clusters $V_1,\dots,V_n$.

We will define the random-cluster probability measure $\phi^{J,\zeta,h}_\L$ using a coupling probability measure $\varphi^{J,\zeta,h}_\L$ defined on $\Si^\zeta_\L\times\Om_\L$.  The marginal distribution of $\varphi^{J,\zeta,h}_\L$ on $\Si^\zeta_\L$ will be the Ising measure $\mu^{J,\zeta,h}_\L$.  The marginal distribution of $\varphi^{J,\zeta,h}_\L$ on $\Om_\L$ defines the random-cluster measure $\phi^{J,\zeta,h}_\L$.

With reference to \cite[Section 1.4]{G-RC} define the coupled probability measure $\varphi^{J,\zeta,h}_\L$ as follows.  Let $(\s,\om)\in\Si^\zeta_\L\times\Om_\L$. Recall the notation \eqref{average spin} for average spins. Note that $\s(V)=\pm1$ if and only if $\s(x)=\s(y)$ for all $x,y\in V$. We will say that $\s$ is an $\om$-admissible configuration if
\[
\s(V_+)=+1, \q \s(V_-)=-1\q \text{and} \q \s(V_i)=\pm1,\ i=1,\dots,n.
\]
If $\s$ is $\om$-admissible, with reference to \eqref{partition function} let
\begin{align*}
\varphi^{J,\zeta,h}_\L(\{(\s,\om)\})&= \frac{1}{Z^{J,\zeta,h}_\L} \prod_e p_e^{\om(e)} (1 - p_e)^{1 - \om(e)},
\end{align*}
otherwise let $\varphi^{J,\zeta,h}_\L(\{(\s,\om)\})=0$.

For configurations $\om\in\Om_\L$ under which $\pd^+\L$ is connected to $\pd^-\L$, there are no $\om$-admissible configurations.  Let $\disconnected_\L^\zeta\subset\Om_\L$ represent the set of configurations such that the vertices in $\pd^+\L$ are {\em not} connected to the vertices of $\pd^-\L$; $\disconnected^\zeta_\L$ is the support of $\phi^{J,\zeta,h}_\L$.  For $\om\in\disconnected^\zeta_\L$, there are $2^{n(\L,\om)}$ $\om$-admissible configurations and
\begin{align*}
\phi^{J,\zeta,h}_\L \left( \{\omega\} \right)=\frac{2^{n(\L,\om)}}{Z^{J,\zeta,h}_\L} \prod_e p_e^{\om(e)} (1 - p_e)^{1 - \om(e)}.
\end{align*}
The coupling $\varphi^{J,\zeta,h}_\L$ has a probabilistic interpretation. To sample an Ising configuration $\s\sim\mu^{J,\zeta,h}_\L$ given a sample $\om\sim \phi^{J,\zeta,h}_\L$, set $\s(V_+)=+1$, set $\s(V_-)=-1$, and independently for $i=1,\dots,n$ set
\begin{align*}
\s(V_i)=
\begin{cases}
+1 & \text{with probability } 1/2,\\ -1 & \text{otherwise}.\\
\end{cases}
\end{align*}
It is sometimes easier to ignore the ghost edges. Let $r(\om)$ denote the edge configuration obtained by closing all the ghost edges. We will say that $V\subset\L$ is a {\em real} cluster if $V$ is an $r(\om)$-cluster. If $V$ is a real cluster under $\phi^{J,\zeta,h}_\L$, but not connected to $\pd^\pm\L$, then
\begin{align}\label{weights}
\s(V)=
\begin{cases}
+1 & \text{with probability } e^{\b h |V|}/(1+e^{\b h |V|}),\\ -1 & \text{otherwise}.\\
\end{cases}
\end{align}
Let $x\lra y$ denote the event that $x$ and $y$ are in the same real cluster, and let $A\lra B$ denote the event that $a \lra b$ for some $a\in A$ and $b\in B$.
Note that when $h=0$, all the clusters are real clusters.

There are two special cases of the $h=0$ random-cluster model, wired and free boundary conditions.  Wired boundary conditions refers to either all-plus or all-minus boundary conditions.  Under wired boundary conditions $\disconnected^\zeta_\L = \Om_\L$.  The limit $\phi^{J,\w}=\lim_{\L\to\ZZ^d} \phi^{J,+,0}_\L$ is called the wired random-cluster measure on $\ZZ^d$.  Free boundary conditions refers to pretending that $\pd\L=\emptyset$ whilst defining $\phi^{J,\zeta,h}_\L$. The resulting measure $\phi^{J,\f,0}_\L$ depends only on $(J(e):e\in E(\L))$. The measure $\phi^{J,\f}$ obtained by taking the limit of $\phi^{J,\f,0}_\L$ as $\L\to\ZZ^d$ is called the free random-cluster measure on $\ZZ^d$.

Let $0\lra\oo$ denote the event that the origin is in an infinite real-cluster. The set
\begin{align}\label{sN}
\sN:=\left\{\b:\QQ \left[\mu^{J,\f}(0 \lra\oo)\right] < \QQ \left[\mu^{J,\w}(0\lra\oo)\right]\right\}
\end{align}
is at most countable \cite{wouts-coarse-graining}. It is conjectured that $\sN=\emptyset$.

\subsection{Stochastic orderings}

Given two measures $\mu_1,\mu_2$ on a set $\RR^\L$, we will write $\mu_1\lest\mu_2$ if there is a coupling $(\s_1,\s_2)$ on $\RR^\L\times\RR^\L$ such that
\begin{romlist}
\item $\s_1 \sim \mu_1$,
\item $\s_2 \sim \mu_2$, and
\item $\s\le\s'$ with probability one.
\end{romlist}
Holley's inequality \cite[Theorem 2.1]{G-RC} can be used to prove stochastic orderings for the ferromagnetic Ising model.  The Ising model on a fixed graph $\L$ is stochastically increasing with respect to the magnetic field and the boundary conditions: for $h_1\le h_2$ and any $\zeta_1\le \zeta_2$,
\[
\mu^{J,\zeta_1,h_1}_\L\lest \mu^{J,\zeta_2,h_2}_\L.
\]
The effect of expanding the region depends on the boundary conditions. With $\De\subset \L$,
\[
\mu^{J,+,h}_\De \gest \mu^{J,+,h}_\L \qtext{ but } \mu^{J,-,h}_\De \lest \mu^{J,-,h}_\L.
\]
Under plus boundary conditions, the random-cluster representation increases with $h\in[0,\oo)$ and $J$,
\[
\phi^{J_1,+,h_1}_\L\lest \phi^{J_2,+,h_2}_\L \qtext{if} h_2\ge h_1 \ge 0 \text{ and } \forall e,\ J_2(e)\ge J_1(e) \ge 0.
\]
Note that sending $J(e)$ to zero or infinity on the boundary allows us to compare free and wired boundary conditions.

\subsection{Coarse graining}\label{sec:cg}

Coarse graining is an important technique in the study of percolation and the random-cluster model.  The open edges of the random-cluster model percolate for $\b>\bc$.
Slab percolation is a stronger property than percolation \cite{wouts-coarse-graining}.  In three and higher dimensions, slab percolation refers to percolation in a {\em slab} $\ZZ^{d-1}\times\{1,\dots,n\}$.  In two dimensions it refers to the existence of spanning clusters in rectangles with arbitrarily high aspect ratios. The slab-percolation threshold is defined
\begin{align*}
\hbc=\inf \{ \b>0 : \text{ slab percolation occurs under } \QQ[\mu^{J,\f}]\} \ge \bc.
\end{align*}
It is conjectured that $\hbc=\bc$.

With $K$ a positive integer, let
\[
\BB_K = [-K/2, K/2)^d \cap \ZZ^d.
\]
For $i\in\ZZ^d$, let $\BB_K(i) := \BB_K + K i$ denote a copy of $\BB_K$ centered at $K i$.

Let $\L\subset\ZZ^d$. To allow unusually high dilution on the edge boundary of $\L$, let $J\sim\QQ$ and let $J'\in\{0,1\}^E$ denote a set of coupling strengths that agrees with $J$ in $E(\L)$. Let $\om\in\Om_\L$ denote an edge configuration for $\L$.

If $\BB_K(i)\subset \L$ and, looking at the restriction of $\om$ to $\BB_K(i)$, if there is a unique real-cluster $A\subset \BB_K(i)$ connecting the $2d$ faces of $\BB_K(i)$, let $\BB_K^\dagger(i)=A$; let $\BB_K^\ddagger(i)$ denote the real $\L$-cluster containing $\BB_K^\dagger(i)$. Otherwise, let $\BB_K^\dagger(i)=\BB_K^\ddagger(i)=\emptyset$.

Let $\epsCG>0$. Recall the definition \eqref{m star} of the spontaneous magnetization $m^*$.

\begin{definition}
A box $\BB_K(i) \subset \L$ is $\epsCG${\em-good} if:
\begin{romlist}
\item $\BB_K^\dagger(i)$ is connected (by paths of length one) to each $\BB_K^\dagger(j)$ such that $\BB_K(j)\subset\L$ and $\|i-j\|_1=1$.
\item The diameters of the real-clusters of $\L\sm(\cup_j \BB_K^\dagger(j))$ intersecting $\BB_K(i)$ are at most $K/2$.
\item $\BB_K^\ddagger(i)\cap \BB_K(i)$ contains between $K^dm^*(1-\epsCG)$ and $K^dm^*(1+\epsCG)$ vertices.
\end{romlist}
Otherwise $\BB_K(i)$ is $\epsCG${\em-bad}.
\end{definition}

Note that any box not entirely contained in $\L$ is automatically $\epsCG$-bad.
If a box is $1$-bad then it is also $\epsCG$-bad for all $\epsCG\in(0,1)$.
Recall that we have fixed $\b > \b_0\ge\hbc$.
The supremum below is over $J'\in\{0,1\}^E$ that agree with $J$ on $E(\L)$.
Combining \cite[Theorem~2.1 and Proposition~2.2]{wouts-coarse-graining} yields:

\begin{proposition}
\label{prop:coarse graining}
There are constants $\cg=\cg(\epsCG)>0$ and $K_0=K_0(\epsCG)\in\NN$ such that for $K\ge K_0$ and any $\BB_K(i_1),\dots,\BB_K(i_n)\subset\L$, with $\QQ$-probability $1-\exp(-\cg K n)$,
\begin{align*}
\sup_{J'} \varphi^{J',+,0}_\L (\BB_K(i_1),\dots,\BB_K(i_n) \text{ are $\epsCG$-bad}) \le \exp(-\cg K n).
\end{align*}
\end{proposition}

Coarse graining gives a crude measure of the cost of phase coexistence.
Consider a path of neighboring boxes $\BB_K(i_1),\dots,\BB_K(i_j)$ in $\L$.
If the first box and last box are not connected by a path of open edges, then there must be a $1$-bad box somewhere along the path of boxes. Moreover, there must be a surface of at least $\lfloor K/K_0(1)\rfloor^{d-1}$ $1$-bad $\BB_{K_0(1)}$-boxes separating $\BB_K(i_1)$ from $\BB_K(i_j)$ \cite[cf. Lemma 4.2]{Stability-of-interfaces}. We obtain the following corollary to Proposition~\ref{prop:coarse graining}. Let $\cgf$ denote a positive constant, independent of $\epsCG$.

\begin{corollary}\label{cor:coarse graining}
Let $K\ge K_0(1)$. For $k=1,\dots,n$, let $\BB_K(i^k_1),\dots,\BB_K(i^k_{j_k})$ denote a simple path of neighboring boxes in $\L$ with length $j_k\le \exp(\sqrt K)$. Assume the $n$ chains are disjoint. Let $A$ denote the event that for each $k$, $\BB_K(i^k_1)$ is not connected to $\BB_K(i^k_{j_k})$ in $\cup_{l=1}^{j_k} \BB_K(i^k_l)$. With $\QQ$-probability $1-\exp(-\cgf K^{d-1}n)$,
\[
\sup_{J'} \varphi^{J',+,0}_\L (A) \le \exp(-\cgf K^{d-1}n).
\]
\end{corollary}

We can quantify the extent to which a magnetic field and mixed boundary conditions affect the coarse graining property. Recall that $E^\pm(\L)$ denotes the set of edges connecting $\L$ to the external vertex boundary $\pd^\pm\L$.

\begin{lemma}\label{bad boxes mixed bc}
Let $\zeta\in\Si$.  For any $\BB_K(i_1),\dots,\BB_K(i_n)\subset\L$, with $\QQ$-probability $1-\exp(- \cg K n)$,
\begin{align*}
\sup_{J'} \varphi^{J',\zeta,h}_\L(\BB_K(i_1),\dots,\BB_K(i_n) \text{ are $\epsCG$-bad})& \\\le \exp(\b |E^\pm(\L)| + \b h |\L| - \cg K n)&.\nonumber
\end{align*}
\end{lemma}

\begin{proof}
Let $(\s,\om)$ be an element of the support of $\varphi^{J',+,0}_\L$; under $\om$ no ghost edges are open.
Let $B$ denote the set of configurations that agree with $(\s,\om)$ as far as the vertices and the real edges are concerned,
\[
\frac{\varphi^{J',\zeta,h}_\L(B)}{\varphi^{J',+,0}_\L(\{(\s,\om)\})} \le \frac{Z^{J',+,0}_\L}{Z^{J',\zeta,h}_\L}.
\]
The right-hand side is bounded above by $\exp(\b|E^\pm(\L)|+\b h |\L|)$. The lemma follows by Proposition~\ref{prop:coarse graining}.
\end{proof}

\subsection{The graphical construction of the Glauber dynamics}\label{sec:heat bath}

For $\xi\in\Si^\zeta_\L$, let $(\s^{s,\xi}_t)_{t\ge s}$ denote the dynamic Ising model, started at time $s$ in state $\xi$ and evolving according to the Glauber dynamics (also knows as heat-bath dynamics).
The Glauber dynamics can be described by the following graphical construction.
Let $\s^{s,\xi}_s=\xi$. Place a rate-one Poisson process at each vertex. Label the points of the Poisson process $(x_i,t_i)$ with $t_1<t_2<\dots$; to each point $(x_i,t_i)$ attach a uniform $[0,1]$ random variable $U_i$.
Let $\s^{s,\xi}_{t-}$ denote the Ising configuration immediately before time $t$.
For each point of the Poisson process, resample the spin at $x_i$ from the Ising measure conditional on the state of the neighboring spins.
If the probability $\s(x_i)=+1$ conditional on $\{\s(y)=\s^{s,\xi}_{t_i-}(y):y\sim x\}$ is $q$, set $\s^{s,\xi}_{t_i}(x_i)=+1$ if $U_i>1-q$, and $-1$ otherwise.
The dynamics are:
\begin{romlist}
\item Monotonic, if $\xi<\n$ then $\s^{s,\xi}_t \le \s^{s,\n}_t$.
\item Finite range, to update site $x$ only requires knowledge of the neighbors.
\item Bounded with respect to the transition rates.
\end{romlist}
Our results are also valid for other dynamics, such as the Metropolis dynamics, that share these properties.

\section{$L_1$-theory}\label{sec:L1}

\subsection{Microscopic and mesoscopic scales}
Recall that we have fixed $p\in(\pc,1)$ and $\b>\b_0$ with $\b\not\in\sN$.

To take advantage of the coarse graining result, we introduce some notation.  With reference to Theorem~\ref{theorem:main}, let $h>0$.  Define
\begin{align}\label{K def}
K= \lfloor h^{-1/(2d)} \rfloor, \qq N= K\lfloor h^{-1}/K \rfloor\approx h^{-1}.
\end{align}
We will call $N$ the macroscopic scale. This is the scale at which nucleation of plus droplets occurs.  The number $K$ denotes a mesoscopic scale. We will consider regions with size order $N$ composed of boxes $\BB_K(i)$.

Let $\sD\subset \RR^d$ denote a connected region.  Let $\DD(\sD,N,K)$ denote a discretized version of $\sD$ composed of mesoscopic boxes,
\begin{align}\label{DD}
\DD(\sD,N,K)=\bigcup_{i\in I} \BB_K(i), \qq I=\left\{x\in\ZZ^d:x+\left[-\frac{K}{2N},\frac{K}{2N}\right]^d\subset\sD\right\}.
\end{align}
We have used $h>0$ to define a set $\L=\DD(\sD,N,K)$ with size $\O(N)=\O(h^{-1})$ on which we wish to study the Ising model $\mu^{J,\zeta,h}_\L$. Given $\L$, we will also want to consider the Ising measure $\mu^{J,\zeta,0}_\L$.  From now on, $h$ will always determine the scale $N$ but will not always indicate the strength of the magnetic field.  The coarse graining implies that under $\mu^{J,\f,0}_\L$, $\s(\BB_K(i))$ is close to either $\pm m^*$ with high probability. This motivates the definition below of the magnetization profile $\MM^\zeta_K:\RR^d\to\RR$ associated with the Ising configuration $\s$.

Let $\L$ denote an arbitrary finite subset of $\ZZ^d$; we will mostly be interested in sets of the form $\DD(\sD,N,K)$ but we will also consider sets that differ from $\DD(\sD,N,K)$ around the boundary. Let $\zeta\in\Si$ denote a boundary condition.  Recall the notation \eqref{average spin} and that under $\mu^{J,\zeta,h}_\L$, $\s(x)=\zeta(x)$ for $x\not\in\L$.

For $\s\in\Si_\L^\zeta$, define the profile $\MM_K^\zeta:\RR^d\to\RR$ as follows.  For $\bx\in\RR^d$, choose $i$ such that $N\bx\in[-K/2,K/2)^d+K i$. Let
\begin{align}\label{def MM}
\MM_K^\zeta(\bx)= \frac12\times
\begin{cases}
1+\s(\BB_K(i))/m^*, & \BB_K(i)\subset \L,\\ 1+\s(\BB_K(i)), & \text{otherwise.}\\
\end{cases}
\end{align}
A value of 1 indicates plus phase, 0 indicates minus phase. The idea behind $L_1$-theory is that $\MM_K^\zeta$ can be approximated by the class of bounded variation profiles. The large deviations of $\MM_K^\zeta$ can be described in terms of surface tension.

\subsection{Surface tension in a parallelepiped}

In statistical physics, surface tension is the excess free energy per unit area due to the presence of an interface.  The definitions of surface tension in \cite{schonmann-shlosman} and \cite{wouts-surface-tension} differ by a factor of $\b$; we have chosen to follow \cite{wouts-surface-tension}.

Let $(\bn,\bu_2,\dots,\bu_d)$ denote an orthonormal basis for $\RR^d$ and let $\sR$ denote the rectangular parallelepiped
\begin{align*}
\sR_{L,H}(\bn,\bu_2,\dots,\bu_d):=\left\{t_1 \bn+ \sum_{k = 2}^d t_k\bu_k: \mathbf{t}\in\left[-\frac{H}2,\frac{H}2\right]\times\left[-\frac{L}2,\frac{L}2\right]^{d-1}\right\}.
\end{align*}
$\sR$ is centered at the origin, has height $H$ in the direction $\bn$ and extension $L$ in the other directions.

Let $\L=\DD(\sR,N,1)$ denote a discrete version of $\L$ \eqref{DD}. The box $\L$ has sides of length $NL$ in the directions $\bu_2, \dots, \bu_d$.  The surface tension can be written in terms of either the Ising model partition function \eqref{partition function} or the random-cluster representation.

\begin{definition}\label{def-tauJ}
Let $\zeta$ denote the configuration in $\Si$ given by $\zeta(y)=+1$ if $y\cdot\bn\ge 0$ and $\zeta(y)=-1$ otherwise.  The surface tension $\t^{J}_\L$ is defined by
\begin{align*}
\t^{J}_\L = \frac{1}{(NL)^{d-1}} \log \frac{Z^{J,+,0}_\L}{Z^{J,\zeta,0}_\L} =\frac{1}{(NL)^{d-1}} \log \frac{1}{\phi^{J,+,0}_\L(\disconnected^\zeta_\L)}.
\end{align*}
\end{definition}

Let $J\sim\QQ$. Surface tension converges in probability as $N\to\oo$ \cite[Theorem~1.3]{wouts-surface-tension}:
\begin{proposition}\label{thm-conv-tauq}
For $\b> 0$ and $\bn\in\sS^{d-1}$, there exists $\t(\bn)\ge 0$, the surface tension perpendicular to $\bn$, such that for all parallelepipeds $\sR=\sR_{L,H}(\bn,\bu_2,\dots,\bu_d)$,
\begin{align}\label{def tau}
\t^J_\L \xrightarrow{\QQ\text{-probability}} \t (\bn ) \text{ as } N\to \oo.
\end{align}
\end{proposition}

Surface tension is strictly positive at temperatures below the threshold for slab percolation \cite[Proposition 2.11]{wouts-surface-tension}:

\begin{proposition}
\label{prop-tauq-pos}
There are constants $C,c>0$ such that for $\bn\in\sS^{d-1}$ and $\b > \hbc$, $c \t(\be_1) \le \t (\bn )\le C\t(\be_1)\le C \b$.
\end{proposition}

We note for completeness that we are discussing {\em quenched} surface tension. {\em Annealed} surface tension, which will not be used in this paper, describes the cost of phase coexistence under the averaged measure $\QQ \mu^{J,\zeta,h}$.  When studying large deviations under the annealed measure, the environment $J$ changes to reduce the surface tension. Although we are interested in the large deviations of $J$, we prefer to control them `by hand' using the $\QQ_\th$ notation defined in \eqref{Q-th-def}. This is less efficient in terms of the size of the large deviation needed. However, it is much simpler.

\subsection{Surface tension in cones}
\label{sec:st}

$L_1$-theory describes the Ising model at equilibrium \cite{BIV1,CerfStFlour,wouts-surface-tension}.
We will restrict our attention to certain subsets of $\RR^d$ with zero surface tension at the boundary. At the microscopic scale, this corresponds to the sampling $J\sim\QQ$ but conditioned on the existence of a surface of edges with $J(e)=0$.

For $\th\in(0,2\pi]$ define a linear cone $\sA_\th$,
\[
\sA_\th:=\{\bx\in\RR^d: x_1 \ge \|\bx\|_2 \cos (\th/2)\}.
\]
For $\th=2\pi$, $\sA_\th$ is simply the whole of $\RR^d$.
Let $\sL^d$ denote the Lebesgue measure on $\sA_\th$ and let $\sV(u,\eps)$ denote the $\sL^d$-ball of radius $\eps$ about $u:\sA_\th\to\RR$.

The {\em perimeter} $\sP(U)$ of a Borel subset $U \subset \sA_\th$ can be written in terms of functions of bounded variation (see \cite[Chapter 3]{N129} and \cite[Section 3.1]{wouts-surface-tension}).
The set of bounded variation profiles $\BV$ is given by
\begin{gather*}
\BV := \left\{ U: U \subset \sA_\th \text{ is a Borel set and } \sP(U) <\oo \right\}.
\end{gather*}
Bounded variation profiles $U \in \BV$ have a {\em reduced boundary} $\pd^\star U$ and an outer normal $\bn^U:\pd^\star U \to \sS^{d-1}$.
Let $\sH^{d-1}$ denote the $d-1$ dimensional Hausdorff measure on $\sA_\th$; $\sH^{d-1} (\pd^\star U) =\sP(U)$.
We will write $\pd U$ to refer to the reduced boundary of $U$ excluding (when $\th<2\pi$) the boundary of $\sA_\th$,
\begin{align}\label{pdU}
\pd U := \pd^\star U\sm\pd^\star\sA_\th.
\end{align}
The outer normal $\bn^U$ defined on $\pd U$ is Borel measurable.
With reference to \eqref{def tau}, this allows us to define the {\em surface tension} and {\em energy} of bounded variation profiles for the dilute Ising model.
Define the surface tension $\sF$ by
\begin{align}
  \sF (U) = \int_{\pd U} \t (\bn^U(\bx))
  d\sH^{d-1} (\bx) \label{eq-def-Ftau} ,\qq U \in \BV.
\end{align}
Define the energy $\sE$ by
\begin{align}\label{def sE}
\sE(U)=\sF(U) - \b m^* \sL^d(U), \qq U \in \BV.
\end{align}
The motivation for these quantities is that they measure, in the following sense, the cost of phase coexistence associated with the Ising model.
Sample $J$ from $\QQ$ conditional on $J(e)=0$ for all $e=\{x,y\}$ such that $x$ but not $y$ is in $\DD(\sA_\th,N,K)$ \eqref{DD}. This makes the surface tension on $\pd^\star\sA_\th$ zero.
Let $\sD$ denote a compact subset of $\sA_\th$.
For a profile of bounded variation $U\subset\sD$, the surface of $\DD(U,N,K)$ has size $\O(1/h^{d-1})$ and $\DD(U,N,K)$ has volume $\O(1/h^d)$ \eqref{K def}.
Heuristically, we expect that the probability of seeing the plus phase in $\DD(U,N,K)$ and the minus phase in $\DD(\sD\sm U,N,K)$ to be approximately
\begin{align*}
\exp\left(-\frac{\sF(U)}{h^{d-1}} \right)  &\qtext{under} \mu^{J,-,0}_{\DD(\sD,N,K)}, \q\text{ and}\\
\exp\left(\frac{1}{h^{d-1}} \left[\inf_{U'} \sE(U')-\sE(U) \right]\right)  &\qtext{under} \mu^{J,-,h}_{\DD(\sD,N,K)}.
\end{align*}
The infimum is over profiles $U'\in\BV$ compatible with the boundary conditions.
For general $\sD$, it is difficult to evaluate the infimum---the conflicting contributions of the positive field and negative boundary conditions may lead to a complicated equilibrium magnetization profile under $\mu^{J,-,h}_{\DD(\sD,N,K)}$.
The problem is simpler if $\sD$ is the Wulff shape.

\section{The Wulff shape in $\sA_\th$}\label{sec:wulff}

\subsection{Wulff, Winterbottom and Summertop shapes}\label{sec:wulff and summertop shapes}

Let $U\subset\sA_\th$ denote a set of bounded variation. Consider the problem of minimizing $\sF(U)$ given that $U$ has volume $b^d$.

\begin{proposition}\label{summertop prop}
Let $\th\in(0,\pi]\cup\{2\pi\}$.
The problem of finding a set of bounded variation $U\subset \sA_\th$ with volume $b^d$ and minimal surface tension has a unique solution when $\th<\pi$; for $\th=\pi$ and $\th=2\pi$ the solution is unique up to translations.
There is a scaling constant $w_\th$ such that the solution is the convex shape
\begin{align}\label{wulff shape}
\sW_\th(b)=w_\th b \{\bx\in\sA_\th: \forall \bn\in\sS^{d-1},\ \bx\cdot \bn \le \t(\bn)\}.
\end{align}
\end{proposition}
\noindent If we omit the $b$, take $b=w_\th^{-1}$ so that $\sW_\th\equiv\sW_\th(w_\th^{-1})$.

Special cases of $\sW_\th$ are known by a variety of names.
When $\th=2\pi$, $\sW_\th$ is the Wulff shape.
When $\th=\pi$, $\sW_\th$ is the Winterbottom shape.
When $\th\in(0,\pi)$ and $d=2$, $\sW_\th$ is the Summertop shape \cite{Summertop}.
We will refer to $\sW_\th$ as the Wulff shape in $\sA_\th$.

\begin{proof}[Proof of Proposition~\ref{summertop prop}]
Let $U$ denote a compact subset of $\sA_\th$.
With reference to \cite[(G)]{KoteckyPfister}, the formula \eqref{eq-def-Ftau} that defines the surface tension $\sF(U)$ is equivalent to
\[
\sF(U)=\lim_{\eps\to 0} \frac{|(U+\eps \sW_{2\pi})\cap \sA_\th| - |U|}{\eps}.
\]
We can use \eqref{wulff shape} to define a second measure of surface tension. Let
\[
\hat{\sF}(U)=\lim_{\eps\to 0} \frac{|U+\eps \sW_\th|-|U|}\eps.
\]
Observe that:
\begin{romlist}
\item
For any $\bx\in\sA_\th$ and $\eps>0$,
\[
\bx+\eps\sW_\th \subset (\bx+\eps \sW_{2\pi}) \cap \sA_\th.
\]
Thus for $U\in\BV$, $\hat{\sF}(U)\le \sF(U)$.
\item
By the Brunn--Minkowski theorem, $\sW_\th(b)$ is the unique shape in $\sA_\th$ (up to translations) with volume $b^d$ and minimal $\hat{\sF}$-surface tension.
\item
By the convexity of $\sW_{2\pi}$, when $U=\sW_\th$,
\[
(U+\eps\sW_{2\pi})\cap\sA_\th = U+\eps\sW_\th
\]
and so  $\hat{\sF}(\sW_\th(b))=\sF(\sW_\th(b))$.
\end{romlist}
Therefore any shape with volume $b^d$ in $\sA_\th$ with minimal $\sF$-surface tension must take the shape $\sW_\th(b)$. 
The claim of uniqueness when $\th<\pi$ follows from the fact that for any $\bx\in\sA_\th\sm\{0\}$, $\sF(\bx+\sW_\th(b))>\sF(\sW_\th(b))$.
\end{proof}

\subsection{Critical droplets}\label{sec:critical droplets}

Let $\mathsf{E}^\th(b)$ account for the cost of filling $\sW_\th(b)$ with the plus phase,
\begin{align}\label{def phi-sW}
\mathsf{E}^\th(b):=\sE(\sW_\th(b))= b^{d-1}\sF(\sW_\th(1))- b^d \b m^*.
\end{align}
The positive term represents the cost of phase coexistence, the negative term represents the benefit of conforming to the magnetic field. Let $\Bcritical$ denote the maximizer of $\mathsf{E}^\th$, and let $B_\mathrm{root}^\th$ denote the positive root of $\mathsf{E}^\th$,
\begin{align}\label{BcB0}
\Bcritical=\frac{d-1}d \frac{\sF(\sW_\th(1))}{\b m^*}, \qq B_\mathrm{root}^\th=\frac{\sF(\sW_\th(1))}{\b m^*}.
\end{align}
The significance of $B_\mathrm{root}^\th$ is that the Ising measure with minus boundary conditions and magnetic field $h$ on $\DD(\sW_\th(b),N,K)$ favors the minus phase if $b<B_\mathrm{root}^\th$, whereas it favors the plus phase if $b>B_\mathrm{root}^\th$. Let
\begin{align}
\label{A-def}
\mathsf{E}^\th_\c:=\mathsf{E}^\th(\Bcritical)=\left(\frac{\sF(\sW_\th(1))}d\right)^d \left(\frac{d-1}{\b m^*}\right)^{d-1}.
\end{align}
The maximum $\mathsf{E}^\th_\c$ of $\mathsf{E}^\th$ characterizes the energy needed to create arbitrarily large plus droplets in the cone $\sA_\th$, starting from the minus phase.

As the $J(e)$ are independent we can find regions that resemble, due to high local dilution, $\DD(\sW_\th(b),N,K)$ for any $\th\in(0,\pi)$ and any $b>0$. However, to maximize the number of catalysts we do not want to take $b$ any larger than we have to.

\begin{proposition}\label{catalyst size}
The diameter of the critical droplet is bounded uniformly over $\b>\hbc$ and $\th\in(0,\pi)\cup\{2\pi\}$.
As $\th\to 0$, $\mathsf{E}^\th_\c=\O(\b\th^{d-1})$.
\end{proposition}

\begin{proof}
For $U\in\BV$, $\sF(U)$ has order $\b\sH^{d-1}(\pd U)$ by Proposition~\ref{prop-tauq-pos}.
Choose $u_\th$ such that the cone
\[
\sU_\th:= u_\th \{\bx\in\sA_\th:x_1\le 1\}
\]
has unit volume. As $\th\to 0$, both $\sU_\th$ and $\sW_\th(1)$ have length of order $\th^{-(d-1)/d}$ in the $x_1$-direction.
By the optimality of the Wulff shape, $\sF(\sW_\th(1))\le \sF(\sU_\th)$ which has order $\b(\th u_\th)^{d-1}$. Substitute this approximation into \eqref{BcB0} and \eqref{A-def}.
\end{proof}

\subsection{Notation for Wulff shapes}

Consider the discrete analogue $\AA_\th$ of $\sA_\th$ at microscopic scale $N$ and mesoscopic scale $K$ \eqref{DD},
\[
\AA_\th:=\DD(\sA_\th,N,K),
\]
and the discrete analogues of the Wulff shape,
\begin{align*}
\WW_\th(b):=\DD(\sW_\th(b),N,K), \qq b\ge 0.
\end{align*}
For $\th\not=2\pi$, the edge boundary of $\AA_\th$ is infinite, thus the probability of finding a pattern of dilution that carves out a translation of the $\AA_\th$ anywhere in $\ZZ^d$ is zero.
Instead let $\Bmax>0$ denote a fixed, but as yet unknown, quantity.
We will limit our attention to the region $\WW_\th(\Bmax)$.
To impose free boundary conditions on the portion of the boundary of $\WW_\th(\Bmax)$ corresponding to $\pd^\star\sA_\th$, let $\QQ_\th$ denote the dilution measure $\QQ$ conditioned appropriately,
\begin{align}\label{Q-th-def}
\QQ_\th := \QQ [ \,\bcdot \mid \forall x\sim y \text{ such that } x\in\WW_\th(\Bmax) \text{ but } y\not\in\AA_\th,\  J(\{x,y\})=0 ].
\end{align}
The probability of seeing such a pattern of dilution is simply $(1-p)$ raised to the power of the number of edges that are conditioned to be closed in the definition of $\QQ_\th$.
$\sW_\th(b)$ has size order $b\th^{-(d-1)/d}$ along the $x_1$-direction and order $b\th^{1/d}$ in the directions $x_2,\dots,x_d$. The number of edges that need to be diluted is therefore order $(\Bmax \th^{-(d-1)/d})\times(\Bmax \th^{1/d})^{d-2}$.
The $\QQ$-probability of the event conditioned on in \eqref{Q-th-def} is thus
\begin{align}\label{Q-th-def-}
\exp\left(- \left(\Bmax\right)^{d-1}\th^{-1/d} \O\left(\log \tfrac{1}{1-p}\right) \middle/ h^{d-1}\right).
\end{align}
As well as Wulff shaped regions, we also need to consider Wulff `annuli': the difference between two Wulff shapes. With $0\le b_1\le b_2\le\Bmax$, let
\[
\sW_\th(b_1,b_2):= \sW_\th(b_2)\sm \sW_\th(b_1) \qtext{and} \WW_\th(b_1,b_2):= \WW_\th(b_2)\sm \WW_\th(b_1).
\]
When $\th=2\pi$, $\sW_\th(b_1,b_2)$ is annular; otherwise it is simply-connected.
There can be three parts to the boundary of $\sW_\th(b_1,b_2)$:
\begin{romlist}
\item the inner boundary $\pd\,\sW_\th(b_1)$,
\item the outer boundary $\pd\,\sW_\th(b_2)$, and
\item the free part of the boundary $\pd^\star \sA_\th \cap \pd^\star \sW_\th(b_1,b_2)$.
\end{romlist}
If $b_1=0$ then there is no inner boundary and $\WW_\th(b_1,b_2)=\WW_\th(b_2)$. If $\th=2\pi$ then the free part of the boundary is empty.

\begin{figure}
\begin{center}
\begin{picture}(0,0)%
\includegraphics{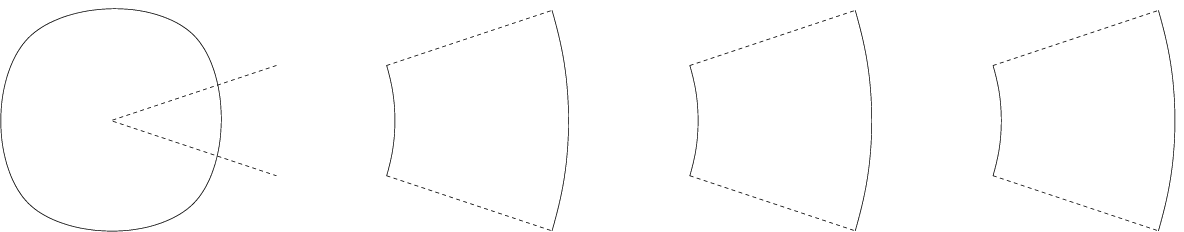}%
\end{picture}%
\setlength{\unitlength}{1160sp}%
\begingroup\makeatletter\ifx\SetFigFont\undefined%
\gdef\SetFigFont#1#2#3#4#5{%
  \reset@font\fontsize{#1}{#2pt}%
  \fontfamily{#3}\fontseries{#4}\fontshape{#5}%
  \selectfont}%
\fi\endgroup%
\begin{picture}(19197,3655)(-911,-77)
\put(3151,2054){\makebox(0,0)[lb]{\smash{{\SetFigFont{8}{9.6}{\familydefault}{\mddefault}{\updefault}$\sA_\th$}}}}
\put(6031,3134){\makebox(0,0)[lb]{\smash{{\SetFigFont{8}{9.6}{\familydefault}{\mddefault}{\updefault}f}}}}
\put(6076,119){\makebox(0,0)[lb]{\smash{{\SetFigFont{8}{9.6}{\familydefault}{\mddefault}{\updefault}f}}}}
\put(5086,1604){\makebox(0,0)[lb]{\smash{{\SetFigFont{8}{9.6}{\familydefault}{\mddefault}{\updefault}$+$}}}}
\put(8371,1649){\makebox(0,0)[lb]{\smash{{\SetFigFont{8}{9.6}{\familydefault}{\mddefault}{\updefault}$-$}}}}
\put(10036,1604){\makebox(0,0)[lb]{\smash{{\SetFigFont{8}{9.6}{\familydefault}{\mddefault}{\updefault}$-$}}}}
\put(13321,1649){\makebox(0,0)[lb]{\smash{{\SetFigFont{8}{9.6}{\familydefault}{\mddefault}{\updefault}$+$}}}}
\put(14986,1604){\makebox(0,0)[lb]{\smash{{\SetFigFont{8}{9.6}{\familydefault}{\mddefault}{\updefault}$+$}}}}
\put(18271,1649){\makebox(0,0)[lb]{\smash{{\SetFigFont{8}{9.6}{\familydefault}{\mddefault}{\updefault}$+$}}}}
\put(10981,3134){\makebox(0,0)[lb]{\smash{{\SetFigFont{8}{9.6}{\familydefault}{\mddefault}{\updefault}f}}}}
\put(11026,119){\makebox(0,0)[lb]{\smash{{\SetFigFont{8}{9.6}{\familydefault}{\mddefault}{\updefault}f}}}}
\put(15931,3134){\makebox(0,0)[lb]{\smash{{\SetFigFont{8}{9.6}{\familydefault}{\mddefault}{\updefault}f}}}}
\put(15976,119){\makebox(0,0)[lb]{\smash{{\SetFigFont{8}{9.6}{\familydefault}{\mddefault}{\updefault}f}}}}
\put(1711,1604){\makebox(0,0)[lb]{\smash{{\SetFigFont{8}{9.6}{\familydefault}{\mddefault}{\updefault}$\sW_\th$}}}}
\put(-269,749){\makebox(0,0)[lb]{\smash{{\SetFigFont{8}{9.6}{\familydefault}{\mddefault}{\updefault}$\sW_{2\pi}$}}}}
\end{picture}%
\end{center}
\caption{From left to right: $\sW_\th$ is the intersection of $\sA_\th$ with $\sW_{2\pi}$. The set $\WW_\th(b_1,b_2)$ with $(+,-)$, $(-,+)$ and $(+,+)$ boundary conditions.  The dotted lines indicate free boundary conditions.\label{fig:bc}}
\end{figure}

Let $(+,-)$ denote an Ising configuration that is equal to $+1$ on $\WW_\th(b_1)$, and equal to $-1$ on $\AA_\th\sm \WW_\th(b_2)$; see Figure~\ref{fig:bc}. We will show in Section~\ref{sec:pcld} that for $b\in[b_1,b_2]$, the probability that $\MM_K^{(+,-)}$ is close to $\sW_\th(b)$ under $\mu^{J,(+,-),0}_\L$ is approximately
\[
\exp\left(\frac{\sF(\sW_\th(b_1))-\sF(\sW_\th(b))}{h^{d-1}} \right).
\]
As well as $(+,-)$ boundary conditions, we will also consider boundary conditions of $(-,+)$, $(+,+)$ and $(-,-)$.
We may simplify $(+,+)$ to + and $(-,-)$ to $-$.

We will also need to consider some sets that only differ from $\WW_\th(b)$ and $\WW_\th(b_1,b_2)$ at the mesoscopic scale.
The sets $\WW_\th(b)$ are subsets of the discrete cone $\AA_\th$.
Let $\{x_1,x_2,\dots\}$ denote an ordering of $\AA_\th$ and let
\begin{align}\label{increasing sets}
\De_\th^n=\{x_1,\dots,x_n\}, \qq n\ge 0.
\end{align}
As a function of $b$, $|\WW_\th(b)|$ is non-decreasing. $\WW_\th(b)$ is composed of boxes of $K^d$ vertices, so $|\WW_\th(\bcdot)|$ is a step function (i.e. piece-wise constant and cadlag).
We can assume that the ordering has been chosen so that for $b\ge 0$, $\WW_\th(b)=\De_\th^{|\WW_\th(b)|}$.
Let $\WW_\th'(\bcdot)$ denote a second family of increasing subsets of $\AA_\th$ such that
\begin{romlist}
\item $\WW_\th(b)=\WW_\th'(b)$ at the points $b$ of discontinuity of $|\WW_\th(b)|$,
\item $\WW_\th(b)\subset \WW_\th'(b)$,
\item for each $n$, for some $b_\th(n)$, $\WW_\th'(b_\th(n))=\De_\th^n$.
\end{romlist}
Let $\WW_\th'(b_1,b_2)=\WW_\th'(b_2)\sm\WW_\th(b_1)$.

\subsection{High and uniformly high $\QQ_\th$-probability}
We will say that an event occurs with {\em high $\QQ_\th$-probability} if under $\QQ_\th$ it occurs with probability at least $1-C\exp(-c/\sqrt h)$.
Note that by taking $C$ large, we only have to consider $h\in(0,h_0)$ where $h_0$ can be arbitrarily small.

Given $\Bmax>0$, suppose that
\begin{align}\label{bees}
0\le b_1\le b\le b_2\le \Bmax \qtext{and} \L=\WW_\th'(b_1,b_2).
\end{align}
Given a class of events defined in terms of $b,b_1$ and $b_2$ we will say that they occur with {\em uniformly high $\QQ_\th$-probability} if each event occurs with probability at least $1-C\exp(-c/\sqrt h)$, uniformly over \eqref{bees}.
Abusing this notation, we may place some additional restriction on $b_1$ and $b_2$; for example fixing $b_1=0$. In that case interpret \eqref{bees} with the additional restriction in place.

\subsection{$L_1$-theory under $(\w,-)$ boundary conditions}

The $L_1$-theory developed in \cite{wouts-surface-tension} describes the dilute Ising model in cubes $\L=\{1,\dots,N\}^d$ under the measure $\mu^{J,-,0}_\L$, $J\sim\QQ$. The proofs in \cite{wouts-surface-tension} are easily adapted to sets of the form $\WW_\th'(b_1,b_2)$ with $J\sim\QQ_\th$. Moreover, the methodology accommodates the $(\w,-)$ boundary conditions described below.

Consider $\L$ as in \eqref{bees}. Let $(\w,-)$ denote wired boundary conditions on the inner boundary of $\L$, and minus boundary conditions on the outer boundary of $\L$. If $b_1=0$, $(\w,-)$ simply means minus boundary conditions.

This is equivalent to starting from $(+,-)$ or $(-,-)$ boundary conditions, and then replacing the inner boundary of $\L$ with a single Ising spin variable,
\begin{align*}
\mu^{J,(\w,-),0}_{\L} (\{\s\})
=\frac{1}{\sum_{s=\pm1}Z^{J,(s,-),0}_\L} \cdot\begin{cases}
\exp  \left( - \b H_{\L}^{J,(+,-),0}(\s)\right), & \s\in\Si^{(+,-)}_\L,\\
\exp  \left( - \b H_{\L}^{J,(-,-),0}(\s)\right), & \s\in\Si^{(-,-)}_\L.\\
\end{cases}
\end{align*}
This is a measure on $\Si^{(+,-)}_\L\cup\Si^{(-,-)}_\L$.
Let $\pd^\w\L$ denote the wired inner-boundary of $\L$. Let $\s(\pd^\w\L)=\pm1$ according to whether $\s\in\Si^{(+,-)}_\L$ or $\s\in\Si^{(-,-)}_\L$.
Define $\MM_K^{(\w,-)}$ to be either $\MM_K^{(+,-)}$ or $\MM_K^{(-,-)}$ according to $\s(\pd^\w\L)$.

The $(\w,-)$ measure with $h=0$ has a natural random-cluster representation with wired-inner and wired-outer boundary conditions. The corresponding coupled measure is
\begin{align}\label{random-cluster w-}
\varphi^{J,(\w,-),0}_\L(\{(\s,\om)\})=\frac{\sum_{s=\pm1} Z^{J,(s,-),0}_\L \varphi_\L^{J,(s,-),0}(\{\s,\om\})}{\sum_{s=\pm1}Z^{J,(s,-),0}_\L}.
\end{align}
Note that only the $s=\s(\pd^\w\L)$ term in the numerator of \eqref{random-cluster w-} is positive.

Let $\disconnected^{(\w,-)}_\L$ refer to the event that the inner boundary $\pd^\w\L$ and the outer boundary $\pd^-\L$ are not connected in the random-cluster representation.
For $\eps>0$, let $\disconnected_{b,\eps}$ denote the event that the inner-boundary takes the plus spin and that the phase profile is in a neighborhood of $\sW_\th(b)$,
\begin{align*}
\disconnected_{b,\eps}:=D^{(\w,-)}_\L \cap\{\s(\pd^\w\L)=1\}\cap \left\{\MM_K^{(\w,-)} \in \sV(\sW_\th(b), \eps)\right\}.
\end{align*}
Parts (i) and (ii) below follow from the proofs in \cite{wouts-surface-tension} of Proposition 3.10 and Theorem 1.11 respectively.

\begin{proposition}
\label{st-results}
Let $b\in[b_1,b_2]$ and $\epsWMbc>0$. With high $\QQ_\th$-probability:
\begin{romlist}
\item   The probability of phase coexistence is bounded below,
\begin{align*}
h^{d-1} \log \varphi^{J,(\w,-),0}_\L \left(\disconnected_{b,\epsWMbc} \right) \ge -\sF (\sW_\th(b)) -\epsWMbc.
\end{align*}
\item
The probability of phase coexistence is bounded above,
\begin{align*}
h^{d-1} \log \mu^{J,(\w,-),0}_\L \left(\int \MM_K^{(\w,-)} \dd\sL^d \ge b^d \right) \le -\sF(\sW_\th(b))+\epsWMbc.
\end{align*}
\end{romlist}
\end{proposition}

\subsection{Large deviations under $(+,-)$ boundary conditions}
\label{sec:pcld}
With $b$ and $\L$ as in \eqref{bees}, we will give upper and lower bounds for the cost of phase coexistence under mixed boundary conditions in the absence of an external magnetic field.

Under $(+,-)$ boundary conditions, the Ising measure favors the minus phase because the minus boundary is bigger. Let $\epsPMbc>0$. 
Large deviations of the magnetization $\MM_K^{(+,-)}$ defined in \eqref{def MM} away from the minus phase are controlled as follows.

\begin{proposition}[Upper bound]
\label{upper bound}
With high $\QQ_\th$-probability
\begin{align*}
h^{d-1} \log \mu^{J,(+,-),0}_\L
\left(
\int \MM_K^{(+,-)} \dd\sL^d
\ge b^d
\right)
\le \sF(\sW_\th(b_1))-\sF(\sW_\th(b))+\epsPMbc.
\end{align*}
\end{proposition}

\begin{proposition}[Lower bound]
\label{lower bound}
With high $\QQ_\th$-probability
\begin{align*}
h^{d-1} \log \mu^{J,(+,-),0}_\L
\left(
\int \MM_K^{(+,-)} \dd\sL^d
\ge b^d
\right)
\ge \sF(\sW_\th(b_1))-\sF(\sW_\th(b))-\epsPMbc.
\end{align*}
\end{proposition}

\begin{proof}[Proof of Proposition~\ref{upper bound}]
By the definition of the $(\w,-)$ measure,
\begin{align}\label{plus-to-wired}
&\mu^{J,(+,-),0}_\L \left(\int \MM_K^{(+,-)} \dd\sL^d \ge b^d\right)\\\nonumber
&\le \mu^{J,(\w,-),0}_\L \left(\int \MM_K^{(\w,-)} \dd\sL^d \ge b^d \right)\cdot
\left(\frac{Z^{J,(+,-),0}_\L}{\sum_{s=\pm1}Z^{J,(s,-),0}_\L}\right)^{-1}.
\end{align}
Proposition~\ref{st-results} part (ii) provides an upper bound on the first term on the right-hand side of \eqref{plus-to-wired},
\[
\mu^{J,(\w,-),0}_\L \left(\int \MM_K^{(\w,-)} \dd\sL^d \ge b^d \right)\le \exp\left(-\frac{\sF(\sW_\th(b))-\epsWMbc}{h^{d-1}}\right).
\]
For $\om\in\disconnected^{(\w,-)}_\L$, the $s=+1$ and $s=-1$ terms in the numerator of the right-hand side of \eqref{random-cluster w-} are equal, so
\begin{align*}
\phi^{J,(\w,-),0}_\L(\om)
&=\frac{2Z^{J,(+,-),0}_\L\phi^{J,(+,-),0}_\L(\om)}{\sum_{s=\pm1}Z^{J,(s,-),0}_\L}.
\end{align*}
Summing over $\om$,
\begin{align}\label{prob-disconnect}
\phi^{J,(\w,-),0}_\L(\disconnected^{\w,-}_\L) = \frac{2Z^{J,(+,-),0}_\L}{\sum_{s=\pm1}Z^{J,(s,-),0}_\L}.
\end{align}
By Proposition~\ref{st-results} part (i) with $b=b_1$,
\begin{align}\label{prob-disconnect2}
\phi^{J,(\w,-),0}_\L(\disconnected^{\w,-}_\L)
\ge \varphi^{J,(\w,-),0}_\L(\disconnected_{b_1,\epsWMbc}) \ge \exp\left(-\frac{\sF(\sW_\th(b_1))+\epsWMbc}{h^{d-1}}\right).
\end{align}
Combining inequalities \eqref{prob-disconnect} and \eqref{prob-disconnect2} produces an upper bound for the second term on the right-hand side of \eqref{plus-to-wired}.
The proposition follows by taking $\epsWMbc=\epsPMbc/3$.
\end{proof}

\begin{proof}[Proof of Proposition~\ref{lower bound}]
With $\epsWMbc>0$ let $b'=\sqrt[d]{b^d+\epsWMbc}$ so that $\disconnected_{b',\epsWMbc}$ implies $\int \MM_K^{(+,-)} \dd\sL^d \ge b^d$ . By equation \eqref{prob-disconnect} and the definition of the $(\w,-)$ measure,
\begin{align*}
\varphi^{J,(+,-),0}_\L \left(\disconnected_{b',\epsWMbc} \right)
= \frac{2 \varphi^{J,(\w,-),0}_\L \left(\disconnected_{b',\epsWMbc}\right)}{ \phi^{J,(\w,-),0}_\L(\disconnected^{(\w,-)}_\L)}.
\end{align*}
Proposition~\ref{st-results} part (i) gives a lower bound for the numerator on the right-hand side. Proposition~\ref{st-results} part (ii) with $b=b_1$ gives an upper bound on the denominator on the right-hand side.
Take $\epsWMbc$ sufficiently small.
\end{proof}

\begin{proposition}
With reference to \eqref{bees}, for fixed $\epsPMbc$ both Proposition~\ref{upper bound} and Proposition~\ref{lower bound} hold with uniformly high $\QQ_\th$-probability.
\end{proposition}
\begin{proof}
Consider Proposition~\ref{upper bound}; the case of Proposition~\ref{lower bound} follows similarly.
Controlling uniformly the $\QQ_\th$-probability can be reduced to the control of a finite number of events.
Let $M(b_1,b_2,b)$ denote the increasing event $\left\{\s:\int \MM_K^{(+,-)} \dd\sL^d \ge b^d\right\}$ with $\MM_K^{(+,-)}$ defined with respect to $\WW_\th'(b_1,b_2)$.

The measure $\mu^{J,(+,-),h}_\L$ is stochastically increasing with $b_1$ and $b_2$.
Let $C,\eps>0$ and let
\[
b_1'=\eps\lceil b_1/\eps\rceil, \q b_2'=\eps\lceil b_2/\eps\rceil,\q b'=\eps\left\lfloor \eps^{-1}\sqrt[d]{b^d-C\eps(\Bmax)^{d-1}} \ \right\rfloor.
\]
For triples $(b_1,b_2,b)$ satisfying \eqref{bees}, the triple $(b_1',b_2',b')$ takes a finite number of values.
With high $\QQ_\th$-probability,
\[
h^{d-1} \log \mu^{J,(+,-),0}_{\WW_\th'(b_1',b_2')}
\left(
M(b_1',b_2',b')
\right)
\le \sF(\sW_\th(b_1'))-\sF(\sW_\th(b'))+\frac\epsPMbc2.
\]
$C$ can be chosen such that if $\sF(\sW_\th(b_1))-\sF(\sW_\th(b))+\epsPMbc<0$ and $\eps$ is sufficiently small then $M(b_1,b_2,b) \implies M(b_1',b_2',b')$ and
\[
\sF(\sW_\th(b_1')) -\sF(\sW_\th(b')) +\frac\epsPMbc2\le \sF(\sW_\th(b_1))-\sF(\sW_\th(b))+\epsPMbc. \qedhere
\]
\end{proof}

\section{Spatial and Markov chain mixing}
\label{sec:technical}

We will describe the dilute Ising model at equilibrium under mixed boundary conditions and a magnetic field.
These results will be used in the next section to show that sufficiently large plus droplets spread in a predictable way.

\subsection{Stability of minimum energy profiles}\label{sec:stab}
Consider the Ising measure with a magnetic field $h$ and $(+,-)$ boundary conditions.
Recall the definitions of the energy functions $\sE$ \eqref{def sE} and $\mathsf{E}^\th$ \eqref{def phi-sW}.
With reference to \eqref{bees}, consider the case $\mathsf{E}^\th(b_1)>\mathsf{E}^\th(b_2)$.
The minimum value of $\mathsf{E}^\th(b)$ is attained when $b=b_2$.
Geometrically, this means that the profile $U$ with $\sW_\th(b_1)\subset U\subset \sW_\th(b_2)$ that minimizes $\sE(U)$ is $\sW_\th(b_2)$.
The minimizer is unique and stable.

\begin{proposition}
\label{lem:stability}
Let $\epsStab\in(0,1)$ such that $b_1< b_2(1-\epsStab)$ and $\mathsf{E}^\th(b_1)>\mathsf{E}^\th(b_2(1-\epsStab))$.
There is a constant $\energyStabilityConstant=\energyStabilityConstant(\epsStab)>0$, independent of $\Bmax$, such that for profiles $U\in\BV$,
\begin{align}
\label{eq:stability}
\sL^d(U)\in[b_1^d,b_2^d (1-\epsStab)^d] \implies \sE(U) \ge \sE(\sW_\th(b_2))+2b_2 \energyStabilityConstant.
\end{align}
Given $\epsStab$, with uniformly high $\QQ_\th$-probability,
\begin{align}
\label{nu bound}
\mu^{J,(+,-),h}_\L \left(\int \MM_K^{(+,-)} \dd \sL^d \le b_2^d (1-\epsStab)^d \right)\le \exp\left( - \frac{b_2\energyStabilityConstant}{h^{d-1}}\right).
\end{align}
\end{proposition}

\begin{proof}
We will first show inequality \eqref{eq:stability} with
\[
\energyStabilityConstant:=\frac{(d-1) \sF(\sW_\th(1)) \left(\Bcritical\right)^{d-2} \eps^2 }{ 4d^2 \sqrt[d]{1-\eps}} \qtext{where} \eps:=1-(1-\epsStab)^d.
\]
By the optimality of the Wulff shape, Proposition~\ref{summertop prop}, we can assume that $U=\sW_\th(b)$ for some $b$.
The unique maximum of $\mathsf{E}^\th$ occurs at $\Bcritical$. As $\mathsf{E}^\th(b_1)>\mathsf{E}^\th(b_2(1-\epsStab))$ the function $\mathsf{E}^\th$ must be decreasing in the region $[b_2(1-\epsStab),b_2]$; it is optimal to consider
\begin{align}\label{b-nu}
b=b_2(1-\epsStab)\ge \Bcritical.
\end{align}
With $b$ as above,
\begin{align}\label{phiphi}
\mathsf{E}^\th(b)-\mathsf{E}^\th(b_2) &= \sF(\sW_\th(1))(b^{d-1}-b_2^{d-1}) - \b m^*(b^d-b_2^d) \\
                  &=\sF(\sW_\th(1))b_2^{d-1}((1-\eps)^{(d-1)/d}-1) + \b m^* \eps b_2^d.\nonumber
\end{align}
By \eqref{BcB0} and \eqref{b-nu},
\begin{align}\label{bc-nu}
\b m^*\eps b_2^d \ge \b m^* \eps \frac{b_2^{d-1} \Bcritical}{\sqrt[d]{1-\eps}} = \sF(\sW_\th(1)) b_2^{d-1}  \frac{\eps}{\sqrt[d]{1-\eps}} \frac{d-1}d.
\end{align}
Let $f(\eps)=(1-\eps)-\sqrt[d]{1-\eps}$. For $\eps\in(0,1)$, $f'''(\eps)>0$ so
\begin{align*}
f(\eps)-f(0)-\eps f'(0) \ge f''(0) \eps^2/2,\q \eps\in[0,1],
\end{align*}
or more explicitly
\begin{align}\label{nu ineq}
(1-\eps)-\sqrt[d]{1-\eps}+\eps\frac{d-1}d \ge \eps^2\cdot \frac{d-1}{2d^2}, \q \eps\in[0,1].
\end{align}
Substituting \eqref{bc-nu} into \eqref{phiphi} and then using \eqref{nu ineq} gives
\begin{align*}
\mathsf{E}^\th(b)-\mathsf{E}^\th(b_2)\ge \sF(\sW_\th(1)) b_2^{d-1} \cdot \frac{ \eps^2 }{\sqrt[d]{1-\eps}} \cdot \frac{d-1}{2d^2} \ge 2b_2\energyStabilityConstant
\end{align*}
as required.

We will now show inequality \eqref{nu bound}.
Let $S$ count, up to an additive constant, the number of plus spins in $\L$,
\[
S:=\sum_{x\in\L} \frac{\s(x)+m^*}2.
\]
By the definition of $\L$, $\MM_K^{(+,-)}$ and $S$,
\[
\left|m^*b_1^d+h^d S - m^* \int \MM_K^{(+,-)} \dd \sL^d\right|=\O(K/N).
\]
The magnetic field corresponds to a Radon-Nikodym derivative controlled by $S$.
With $Z:=\mu^{J,(+,-),0}_\L\left[\exp\left(\b h S  \right)\right]$,
\begin{align*}
\frac{\mu^{J,(+,-),h}_\L( \{\s\} )}{\mu^{J,(+,-),0}_\L(\{\s\})}
= \frac{1}{Z}\exp\left(\b h S \right).
\end{align*}
We will find a lower bound on $Z$.
When $h$ is sufficiently small,
\[
\MM_K^{(+,-)} \in \sV(\sW_\th(b),\epsPMbc) \implies |h^{d}\, S - m^*(b^d-b_1^d)| =\O(\epsPMbc).
\]
Applying Proposition~\ref{lower bound} with $b=b_2$,
\begin{align*}
h^{d-1} \log Z
&\ge h^{d-1} \log \mu^{J,(+,-),0}_\L (\MM_K^{(+,-)} \in \sV(\sW_\th(b_2),\epsPMbc))+\b m^* ( b_2^d-b_1^d)  -\O(\epsPMbc)\\
&\ge \sF(\sW_\th(b_1))-\sF(\sW_\th(b_2))+\b m^* ( b_2^d-b_1^d) -\O(\epsPMbc)\\
&\ge \sE(\sW_\th(b_1))-\sE(\sW_\th(b_2))-\O(\epsPMbc).
\end{align*}
Let $n=\lfloor (b_2(1-\epsStab)-b_1)/\epsPMbc\rfloor$. We can write
\begin{align*}
&\left\{\s:\int \MM_K^{(+,-)} \dd \sL^d \le  b_2^d (1-\epsStab)^d \right\}\subset\\
&\bigcup_{b\in\{b_1,b_1+\epsPMbc,\dots,b_1+\epsPMbc n\}} \left\{\s:\int \MM_K^{(+,-)} \dd \sL^d \in[b^d,(b+\epsPMbc)^d] \right\}.
\end{align*}
By Proposition~\ref{upper bound}, for $b\in\{b_1,b_1+\epsPMbc,\dots,b_1+\epsPMbc n\}$,
\begin{align*}
&\mu^{J,(+,-),0}_\L\left[\exp  \left( \b h S \right);\int \MM_K^{(+,-)} \dd\sL^d\in[b^d,(b+\epsPMbc)^d ]\right]\\
&\le \exp\left(h^{1-d}  \left[\sF(\sW_\th(b_1))-\sF(\sW_\th(b))+ \b m^*   [(b+\epsPMbc)^d-b_1^d]+ \O(\epsPMbc) \right] \right)\\
&\le \exp\left(h^{1-d}  \left[\sE(\sW_\th(b_1))-\sE(\sW_\th(b))+ \O(\epsPMbc)\right]\right).
\end{align*}
The left-hand side of \eqref{nu bound} is therefore at most
\[
\sum_{b\in\{b_1,b_1+\epsPMbc,\dots,b_1+\epsPMbc n\}} \exp\left(h^{1-d} [\sE(\sW_\th(b_2))-\sE(\sW_\th(b))
+\O(\epsPMbc)
]\right).
\]
Taking $\epsPMbc$ small with respect to $\epsStab$, inequality \eqref{nu bound} follows by \eqref{eq:stability}.
\end{proof}

Consider now the case $\mathsf{E}^\th(b_1)<\mathsf{E}^\th(b_2)$.
The optimum profile matching $(+,-)$ boundary conditions is $\sW_\th(b_1)$ so the minus phase is dominant.

\begin{proposition}
\label{lem:stability2}
Let $\epsStab>0$. Suppose that $\mathsf{E}^\th\left(b_1+\epsStab\right)< \mathsf{E}^\th(b_2)$. There is a constant $\energyStabilityConstant'=\energyStabilityConstant'(\epsStab)>0$ such that with uniformly high $\QQ_\th$-probability
\begin{align*}
\mu^{J,(+,-),h}_\L \left(\int \MM_K^{(+,-)} \dd \sL^d \ge (b_1 +\epsStab)^d \right)\le \exp\left(- \frac{\energyStabilityConstant'}{h^{d-1}}\right).
\end{align*}
\end{proposition}

We will omit the proof of Proposition~\ref{lem:stability2} as it is similar to the proof of Proposition~\ref{lem:stability}.

Now consider the case $b_1=0$ and $b_2>\Bcritical$ under minus boundary conditions. In order to get a stability property that does not depend on the sign of $\mathsf{E}^\th(b_2)$ we will condition on seeing a large region of the plus-phase. With $\epsC>0$ let
\begin{align}\label{mu hat}
\hm_\L&:=\mu^{J,-,h}_\L ( \bcdot \mid \sC) \q \text{where} \\
\sC=\sC(\epsC)&:=\left\{\s\in\Si_\L^-:\int_{\sW_\th(\Bcritical+\epsC)} \MM_K^{-}\dd \sL^d \ge (\Bcritical)^d \right\}.\nonumber
\end{align}
\begin{proposition}
\label{lem:stabilityR}
Let $\epsStab,\epsC>0$. Suppose $b_1=0$ and $b_2(1-\epsStab)> \Bcritical+\epsC$.
Recall $\energyStabilityConstant$ from Proposition~\ref{lem:stability}. Given $\epsStab$ and $\epsC$, with uniformly high $\QQ_\th$-probability,
\begin{align*}
\hm_\L \left(\int \MM_K^{-} \dd \sL^d \le b_2^d (1-\epsStab)^d \right)\le \exp \left(- \frac{b_2\energyStabilityConstant}{h^{d-1}}\right).
\end{align*}
\end{proposition}
\begin{proof}
The proof of \eqref{eq:stability} (with $\Bcritical$ playing the role of $b_1$) implies that for all $U\in \BV$,
\begin{align*}
\sL^d(U)\in [(\Bcritical)^d,b_2^d (1-\epsStab)^d] \implies \sE(U) \ge \sE(\sW_\th(b_2))+2b_2 \energyStabilityConstant.
\end{align*}
Therefore $\mathsf{E}^\th(b_2 (1-\epsStab))-\mathsf{E}^\th(b_2)\ge 2b_2\energyStabilityConstant$.

Returning to the context of Proposition~\ref{lem:stabilityR}, we have $b_1=0$.
Let $a=\min\{0,\mathsf{E}^\th(b_2)\}$ denote the minimum of $\mathsf{E}^\th(b)$ for $b\in[b_1,b_2]$.
Let $b$ denote the minimum value in the range $[b_2(1-\epsStab),b_2]$ such that
\[
b_2^d-b^d \le (\Bcritical+\epsC)^d-(\Bcritical)^d.
\]
Treating the magnetic field as a Radon-Nikodym derivative as in the proof of Proposition~\ref{lem:stability},
\begin{align*}
\mu^{J,-,h}_\L\left(\int \MM_K^{-} \dd \sL^d \in \left[(\Bcritical)^d,b_2^d(1-\epsStab)^d\right] \right)\le \exp\left(\frac{a-\mathsf{E}^\th(b_2)-3b_2\energyStabilityConstant/2}{h^{d-1}} \right)
\end{align*}
and
\[
\mu^{J,-,h}_\L (\sC)\ge \mu^{J,-,h}_\L\left(\int \MM_K^{-} \dd \sL^d \in \left[b^d,b_2^d\right] \right) \ge \exp\left(\frac{a-\mathsf{E}^\th(b_2)-b_2\energyStabilityConstant/2}{h^{d-1}}\right).\qedhere
\]
\end{proof}

We will now consider two different boundary conditions. By Proposition~\ref{upper bound}, plus/minus symmetry when $h=0$, and monotonicity, the plus phase is dominant under $\mu^{J,(-,+),h}_\L$.

\begin{proposition}
\label{lem:stability3}
Let $\epsStab>0$. There is a constant $\energyStabilityConstant''=\energyStabilityConstant''(\epsStab)>0$ such that with uniformly high $\QQ_\th$-probability,
\begin{align*}
\mu^{J,(-,+),h}_\L \left(\int_{\sW_\th(b_1,b_2)} \MM_K^{(-,+)} \dd \sL^d \le b_2^d-(b_1+\epsStab)^d \right)\le \exp\left(- \frac{\energyStabilityConstant''}{h^{d-1}}\right).
\end{align*}
\end{proposition}

Finally, consider boundary conditions of plus on the inner boundary and free on the outer boundary.

\begin{proposition}
\label{lem:stability4}
Let $\epsStab>0$ and suppose $b_2^d\ge 2\epsStab$. With uniformly high $\QQ_\th$-probability,
\begin{align*}
\mu^{J,(+,\f),h}_\L \left(\int_{\sW_\th(b_1,b_2)} \MM_K^+ \dd \sL^d \le b_2^d-\epsStab\right)\le \exp\left(- \frac{\epsStab}{h^{d-1}}\right).
\end{align*}
\end{proposition}

\noindent We omit the proof as it is similar to the others in this section.

\subsection{Phase labels in $\AA_\th$}

Lemma~\ref{bad boxes mixed bc} provides a simple measure of the cost of phase coexistence in a region conditional on the boundary. Using phase labels---defined in terms of the coarse graining---to describe the boundary conditions allows for a sharper bound.

Take $\L$ as in \eqref{bees}.
Let $\om\in\Om_\L$, and let $\s\in\Si_\L^\zeta$ denote an $\om$-admissible spin configuration.  Define phase labels in terms of the coarse graining with $\epsCG=1$:
\begin{align}\label{phase labels}
\Psi(i)=
\begin{cases}
+1, & \BB_K(i) \text{ is 1-good and }\s(\BB_K^\dagger(i))=+1,\\
-1, & \BB_K(i) \text{ is 1-good and }\s(\BB_K^\dagger(i))=-1,\\
\phantom{+}0, & \text{otherwise}.\\
\end{cases}
\end{align}
Let $\Gamma=\cup_{i\in I}\BB_K(i)$ denote a subset of $\L$ composed of whole mesoscopic boxes.  Let $\pd I$ denote the set of points $i\in\ZZ^d$ at $L_\oo$-distance $2$ from $I$ such that $\BB_K(i)\subset\AA_\th$. Let $\psi:I\cup\pd I\to\{\pm1,0\}$ denote a label configuration assigning labels to the boxes $\BB_K(i)$.  Let $\mathrm{Int}_\psi$ (respectively $\mathrm{Ext}_\psi$) denote the event that $\Psi(i)=\psi(i)$ for $i\in I$ (respectively $\pd I$).  Given $\psi$, let $f_I^0$ count the number of $0$ labels in $I$.  Let $f_{\pd I}^+,f_{\pd I}^-,f_{\pd I}^0$ count the number of boxes in $\pd I$ with phase labels $+1$, $-1$ and $0$.  Let $f_I^\lra$ count the maximal number of disjoint paths (the maximal flow) inside $I$ between the $+1$ labels and $-1$ labels.

We can find constants $\truncationConstant,\truncationConstantB>0$ such that the following hold.

\begin{lemma}\label{truncation}
With $\QQ_\th$-probability $1-\exp(- \truncationConstantB f_I^\lra K^{d-1})$,
\begin{align*}
\varphi^{J,\zeta,h}_\L (\mathrm{Int}_\psi \cap \mathrm{Ext}_\psi) \le \exp \big( \truncationConstant(f_{\pd I}^0+f_{\pd I}^-)K^{d-1}- \truncationConstantB f_I^\lra K^{d-1}\big).
\end{align*}
\end{lemma}

\begin{lemma}\label{truncation-}
With $\QQ_\th$-probability $1-\exp(- \truncationConstantB[f_I^0 K+f_I^\lra K^{d-1}])$,
\begin{align*}
\varphi^{J,\zeta,h}_\L (\mathrm{Int}_\psi \cap \mathrm{Ext}_\psi) \le \exp\big( \truncationConstant \min\{f_{\pd I}^0+f_{\pd I}^-,f_{\pd I}^0+f_{\pd I}^+\}K^{d-1} &\nonumber\\ + \b h |\Gamma|4^d- \truncationConstantB [f_I^0 K+f_I^\lra K^{d-1}&]\big).
\end{align*}
\end{lemma}

\begin{proof}[Proof of Lemma~\ref{truncation}]
Let $\Gamma_k$ denote the union of the mesoscopic boxes at distance at most $k$ from $\Gamma$,
\[
\Gamma_k=\bigcup \{\BB_K(j):\exists\, i\in I,\ \|i-j\|_\oo\le k\}.
\]
Note that $\Gamma\subset\Gamma_1\subset\Gamma_2$, and for $i\in\pd I$, $\BB_K(i)\subset\Gamma_2\sm\Gamma_1$.
Write $\om=\om_\g\oplus\om_\mathrm{int}\oplus\om_\mathrm{ext}$ where $\om_\g$ is the configuration of the ghost edges and $\om_\mathrm{int}$ is the configuration of the real edges $E(\Gamma_1)$.

We would like to condition on the event $\mathrm{Ext}_\psi$. However, the resulting measure is complicated because $\mathrm{Ext}_\psi$ carries information about not just $\om_\mathrm{ext}$ and $(\s(\BB_K^\dagger(i)):i\in\pd I)$, but also about $\om_\mathrm{int}$.
Let $\mathrm{Ext}'_\psi$ denote the event that $\om_\mathrm{ext}$, and the states of the vertices in $\L\sm{\Gamma_1}$, are compatible with $\mathrm{Ext}_\psi$.
Let $\mathrm{Int}'_\psi$ denote the event that $\om_\mathrm{int}$ is compatible with $\mathrm{Int}_\psi$.

The coupled measure $\varphi^{J,\zeta,h}_\L$ has a property related to the finite energy property of the regular random-cluster model. The probability of edge $e$ being closed, conditional on all the other edge states {\em and} all the spin states, is bounded away from $0$:
\begin{align}\label{finite-energy}
\varphi^{J,\zeta,h}_\L(\om(e)=0 \mid \s \text{ and } (\om(f))_{f\not=e}) \ge 1-p_e.
\end{align}
This tells us the cost of conditioning on edges being closed. Surgically closing edges saves us from having to consider mixed boundary conditions.

Let $T$ denote the event that all edges spanning between a box $\BB_K(i)$ with $i\in\pd I$ and $\psi(i)\le 0$ and a box $\BB_K(j)\subset{\Gamma_1}$ are closed.
By \eqref{finite-energy}, for some constant $\truncationConstant>0$,
\begin{align}\label{trunc3}
  \varphi^{J,\zeta,h}_\L(\mathrm{Int}'_\psi \cap \mathrm{Ext}'_\psi) \le \varphi^{J,\zeta,h}_\L(\mathrm{Int}'_\psi \cap \mathrm{Ext}'_\psi\cap T) \exp\left(\truncationConstant(f_{\pd I}^0 + f_{\pd I}^-) K^{d-1}\right).
\end{align}
Suppose that $\om_\mathrm{ext}$ splits $\L\sm{\Gamma_1}$ into $\om_\mathrm{ext}$-clusters $W_+,W_-,W_1,\dots,W_n$.
We stress that $\om_\mathrm{ext}$ is a partial edge configuration: the $\om_\mathrm{ext}$-clusters in $\L\sm{\Gamma_1}$ may be connected by an $(\om_\g\oplus\om_\mathrm{int})$-path.
Assume that the event $\mathrm{Ext}'_\psi\cap T$ holds. If $\BB_K^\ddagger(i)$, $i\in\pd I$, intersects ${\Gamma_1}$ then $\psi(i)=1$. Without loss of generality we can assume that
\begin{romlist}
\item $W_1,\dots,W_l$ correspond to $\BB_K^\ddagger(i)$ with $i\in \pd I$ and $\psi(i)=+1$,
\item $W_{l+1},\dots,W_m$ correspond to clusters with diameter less than $K/2$,
\item $W_{m+1},\dots,W_n$ (and $W_-$) are not connected to ${\Gamma_1}$.
\end{romlist}
Consider the conditional measure $\varphi^{J,\zeta,h}_\L( \bcdot \mid \mathrm{Ext}'_\psi\cap T,\om_\mathrm{ext})$.
The clusters $W_+$ and $W_1,\dots,W_l$ act as plus boundary conditions. The spins of the clusters $W_{l+1},\dots,W_m$ are unknown so the clusters act as wired boundary conditions. Let $\hat\phi^{h}$ denote the marginal measure on $\om_\mathrm{int}$ corresponding to $\varphi^{J,\zeta,h}_\L( \bcdot \mid \mathrm{Ext}'_\psi\cap T,\om_\mathrm{ext})$,
\begin{align}\label{trunc4}
\varphi^{J,\zeta,h}_\L(\mathrm{Int}'_\psi \cap \mathrm{Ext}'_\psi \cap T)\le \sup_{\om_\mathrm{ext}\in\mathrm{Ext}'_\psi\cap T} \hat\phi^h(\mathrm{Int}'_\psi).
\end{align}
Let $A$ denote the decreasing event that there are no $\om_\mathrm{int}$-open paths in $\Gamma_1$ between $\BB_K$-boxes with $\psi(i)=+1$ and $\psi(i)=-1$; note that $\mathrm{Int}'_\psi\subset A$.  By monotonicity, as $A$ is a decreasing event and $\hat\phi^h$ is increasing with $h\in[0,\oo)$,
\begin{align}
\hat\phi^h(\mathrm{Int}'_\psi)\le \hat\phi^h(A)\le\hat\phi^0(A).
\end{align}
With reference to Corollary~\ref{cor:coarse graining}, we can find a collection of $f_I^\lra$ disjoint chains of boxes such that $A$ implies that for each chain, the first box is not connected to the last box. With $\QQ_\th$-probability $1-\exp(- \cgf f_I^\lra K^{d-1})$,
\begin{align}\label{trunc5}
\hat\phi^0(A)\le \exp(-\cgf f_I^\lra K^{d-1}).
\end{align}
Collecting \eqref{trunc3}-\eqref{trunc5} gives the lemma.
\end{proof}

\begin{proof}[Proof of Lemma~\ref{truncation-}]
We will first consider the case $f_{\pd I}^+\ge f_{\pd I}^-$.
Let ${\Gamma_1}$, $\mathrm{Int}'_\psi,\mathrm{Ext}'_\psi$, $T$, $W_1,\dots,W_n$ and $\hat\phi^h$ be defined as above.

Recall inequality \eqref{trunc4}. With reference to \eqref{weights}, the sizes of the $W_{l+1},\dots,W_m$ affect the interaction (under $\hat\phi^h$) of open clusters in ${\Gamma_1}$ with the external magnetic field.
Recall that the clusters $W_{l+1},\dots,W_m$ have diameter at most $K/2$:
\begin{align}\label{trunc4-}
\hat\phi^h(\mathrm{Int}'_\psi)  / \hat\phi^0(\mathrm{Int}'_\psi) \le \exp(\b h &|\Gamma_1\cup W_{l+1}\cup\dots\cup W_m|) \le \exp(\b h |\Gamma| 4^d).
\end{align}
By Proposition~\ref{prop:coarse graining} and Corollary~\ref{cor:coarse graining}, there is a positive constant $c>0$ such that with $\QQ_\th$-probability $1-\exp(-c\, \max\{f_I^0 K,f_I^\lra K^{d-1}\})$,
\begin{align}\label{trunc7}
\hat\phi^0(\mathrm{Int}'_\psi)\le \exp(-c\, \max\{f_I^0 K,f_I^\lra K^{d-1}\}).
\end{align}
Collecting \eqref{trunc3}-\eqref{trunc4} and \eqref{trunc4-}-\eqref{trunc7} gives the lemma in the case $f_{\pd I}^+\ge f_{\pd I}^-$. The proof in the case $f_{\pd I}^->f_{\pd I}^+$ follows by swapping $+$ and $-$ in the definition of $T$ and $\hat\phi^h$.
\end{proof}

\subsection{Hausdorff stability of random-cluster boundaries}\label{sec:hausdorff}

Consider the context of Proposition~\ref{lem:stability}: $\mathsf{E}^\th(b_1)>\mathsf{E}^\th(b_2)$.
Under $\mu^{J,(+,-),h}_\L$ the plus phase is dominant so the minus boundary does not affect the bulk of the domain.
In this section, we will show that, in a random-cluster sense, the $\pd^-\L$ boundary-cluster is small.

\begin{proposition}
\label{decay-bc}
Let $\epsHS>0$. Suppose that $\mathsf{E}^\th(b_1)>\mathsf{E}^\th(b_2(1-\epsHS))$.
There is a constant $\stabilityConstant=\stabilityConstant(\epsHS)>0$, independent of $\Bmax$, such that with uniformly high $\QQ_\th$-probability
\begin{align*}
\phi^{J,(+,-),h}_\L(\pd^-\L\lra\WW_\th(b_1,b_2(1-\epsHS))) \le \exp (-\stabilityConstant b_2 /h).
\end{align*}
\end{proposition}
\noindent
The proof develops the technique of truncation used in \cite{Stability-of-interfaces}. The differences in geometry, and the presence of a magnetic field, pose extra challenges.

\begin{proof}[Proof of Proposition~\ref{decay-bc}]
Let $w$ denote the minimum $L_\oo$-distance between $\pd \sW_\th(1)$ and the origin,
\begin{align}\label{w}
w=\inf_{\bx\in\pd\sW_\th(1)} \|\bx\|_\oo.
\end{align}
Let $R=\lfloor \epsHS b_2 N w/(8 K) \rfloor$ and let $S=\lfloor (b_2-b_1) N w / (2K)\rfloor$. Define mesoscopic layers
\begin{align*}
\HH_l=\WW_\th\left(b_2-\frac{2l K}{wN},\  b_2-\frac{2(l-1) K}{wN}\right), \qq l=1,\dots,S.
\end{align*}
The layers divide the annulus \mbox{$\WW_\th(b_1,b_2)$} into $S$ layers.
Let $\HH_{j,k}=\cup_{l=j}^k \HH_l$.
The annulus $\WW_\th(b_2(1-\epsHS),b_2)$ corresponds to the first $4R$ layers, $\HH_{1,4R}$, with $\HH_1$ the outermost layer.
The layers are essentially $d-1$ dimensional, resembling scalar multiples of $\pd\sW_\th$; Figure \ref{fig:partIII} shows the case $d=2$.
The layers have been constructed so that:
\begin{romlist}
\item The mesoscopic boxes in $\HH_l$ form a surface separating $\HH_{l-1}$ from $\HH_{l+1}$.
\item Each $\HH_l$ is between $2$ and $2d$ mesoscopic boxes thick.
\end{romlist}
Define mesoscopic phase labels $\Psi=(\Psi(i):\BB_K(i)\subset\HH_{1,S})$ according to \eqref{phase labels}.
We will show that with high probability, a surface of $+1$ boxes separates $\HH_1$ from $\HH_{4R}$.

Given a phase label $\psi:\{i:\BB_K(i)\subset\HH_{1,S}\}\to\{\pm1,0\}$, define the profile $\fv$ of $\psi$ as follows.
For $s\in\{-1,0,1\}$, let $f_l^s$ count the number of $s$-boxes in $\HH_l$,
\[
f_l^s=\#\{i:\BB_K(i)\subset\HH_l \text{ and } \psi(i)=s\}.
\]
Let $\fv=(f_l^s)_{l=1,\dots,S}^{s=-1,0,1}$.
We will write $\sF(\fv)$ to denote the set of phase labels $\psi$ compatible with $\fv$.
The number of configurations of the phase labels $\psi$ compatible with $\fv$ is limited by the definition of the coarse graining.
Surfaces of $0$-boxes must separate the plus-boxes from the minus-boxes.
The surfaces of $0$-boxes cannot separate $\HH_l$ into more than $f_l^0+1$ connected components.
Therefore the number of ways of assigning the labels in layer $l$ is bounded by
\begin{align}\label{size of sFfv}
\binom{|\HH_l|K^{-d}}{f_l^0} \exp([f_l^0+1] \log 2) =\exp(f_l^0 \O(\log N)).
\end{align}
We will say that a profile $\fv$ is {\em spanning} if $f_l^{-1}+f_l^0>0$ for $l=1,\dots,4R$.
Recall from \eqref{K def} that $K\approx N^{1/(2d)}$. The number of spanning profiles is less than
\[
\prod_{l=1}^S \left(\frac{|\HH_l|}{K^d}\right)^3 = \O\left(\left(N/K\right)^{d-1}\right)^{\O(N/K)}
\]
which grows more slowly than $\exp(cN)$ for every positive constant $c$.
It is therefore sufficient to find a constant $\stabilityConstant=\stabilityConstant(\epsHS)>0$ and a $\QQ_\th$-event $\sJ$ such that
\begin{align}\label{spanningProfilesConstant claim}
\sJ\subset\left\{J:\forall \fv \text{ spanning, } \varphi^{J,(+,-),h}_\L(\Psi \in \sF(\fv))\le \exp(-2\stabilityConstant b_2 N)\right\}
\end{align}
and $\sJ$ has uniformly high $\QQ_\th$-probability.
The event $\sJ$ is defined as follows.
Let $\sJ=\emptyset$ for $h\ge h_0$ (for some $h_0>0$) so we can assume that $h$ is arbitrarily small.
We will appeal below to Proposition~\ref{lem:stability} and Lemmas \ref{bad boxes mixed bc}, \ref{truncation} and \ref{truncation-}.
For $h<h_0$, let $\sJ$ denote the intersection of the associated $\QQ_\th$-events.

Let $\a,\hat{\a}>0$. Consider three constraints on the label profile:
\begin{align}
\label{volume constraint}
\sum_{l=1}^S f_l^0    &\le N^{d-1}/\sqrt K,\\
\label{volume constraint3}
\sum_{l=1}^S f_l^{-1} &\le \a \left(\frac{b_2N}{K}\right)^d,\\
\label{volume constraint2}
\max_{l=R,\dots,3R} f_l^{-1} &\le \hat{\a}\left(\frac{b_2N}{K}\right)^{d-1}.
\end{align}
For $J\in\sJ$ we will check the inequality in \eqref{spanningProfilesConstant claim} in four parts. We will show:
\begin{Romlist}
\item For any $\stabilityConstant>0$, if \eqref{volume constraint} fails then the inequality in \eqref{spanningProfilesConstant claim} holds.
\item For any $\a>0$, if \eqref{volume constraint} holds but \eqref{volume constraint3} fails, then the inequality in \eqref{spanningProfilesConstant claim} holds if $\stabilityConstant$ is sufficiently small.
\item For any $\hat{\a}>0$, if \eqref{volume constraint} and \eqref{volume constraint3} holds but \eqref{volume constraint2} fails, then the inequality in \eqref{spanningProfilesConstant claim} holds provided that $\a$ and $\stabilityConstant$ are sufficiently small.
\item If \eqref{volume constraint}-\eqref{volume constraint2} hold, then the inequality in \eqref{spanningProfilesConstant claim} holds if $\a,\hat{\a}$ and $\stabilityConstant$ are sufficiently small.
\end{Romlist}
For part (I), choose $\fv$ such that \eqref{volume constraint} fails.
Inequality \eqref{volume constraint} is an upper bound on the volume of bad boxes.
We will apply Lemma~\ref{bad boxes mixed bc} for $\psi\in\sF(\fv)$ with $\epsCG=1$ and $n=\sum_{l=1}^S f_l^0>N^{d-1}/\sqrt{K}$.
The positive terms in the exponential in Lemma~\ref{bad boxes mixed bc} are
\[
\b |\pd^\pm\L| + \b h  |\L| = \O(b_2^d N^{d-1}).
\]
The absolute value of the negative term is greater than $\cg K (N^{d-1}/\sqrt{K})$ which is a higher order of $N$ \eqref{K def}.
For $h$ sufficiently small,
\[
\varphi^{J,(+,-),h}_\L(\Psi=\psi)\le \exp(-\cg n K/2),
\]
and so by \eqref{size of sFfv},
\[
\varphi^{J,(+,-),h}_\L(\Psi\in\sF(\fv))\le \exp\left(n \O(\log N)-\cg nK/2\right).
\]
Whatever the value of the constant $\stabilityConstant>0$, for $h$ sufficiently small the right-hand side above is less than $\exp(-2\stabilityConstant b_2 N)$.

Now to part (II). Inequality \eqref{volume constraint3} is an upper bound on the volume of minus phase.
We can expect the majority of boxes to have label $+1$ because of Proposition~\ref{lem:stability}.

Recall the symbol $\epsCG$ used in the definition of the coarse graining. Let $\psi$ denote a label configuration.
Let $n^+$ count the number of phase labels $\psi(i)=+1$ such that $\BB_K(i)$ is also $\epsCG$-good. Similarly for $n^-$. Let $n^0$ count the number of $\epsCG$-bad boxes in $\HH_{1,S}$.
Note that
\[
\sum_{l=1}^S f_l^{-1} \le n^0+n^-.
\]
Let $M(\epsStab)$ refer to the $\mu^{J,(+,-),h}_\L$-event in Proposition~\ref{lem:stability}.
Under $M(\epsStab)$,
\begin{align}\label{fajskfh1}
\frac{1}{N^d m^*} \sum_{x\in\L}\s(x)\ge b_2^d-b_1^d-2\epsStab b_2^d+\O(K^{-1}).
\end{align}
For each box $\BB_K(i)$ counted by $n^-$ there are at least $(1-\epsCG)K^dm^*$ vertices in $\BB_K^\ddagger(i)\cap\BB_K(i)$.
For each box $\BB_K(i)$ counted by $n^+$ there are at most $(1+\epsCG)K^dm^*$ vertices in $\BB_K^\ddagger(i)\cap\BB_K(i)$.
There are at most $n^0K^d$ vertices in $\epsCG$-bad boxes.
The remaining vertices lie in clusters with diameter less than $K/2$. Small clusters only interact weakly with the magnetic field.
If an open cluster has volume of less than $(K/2)^d$ then the odds of it taking plus spin are at most $\exp(\b h(K/2)^d)$ to $1$ \eqref{weights}.
Conditional on $\Psi=\psi$, with probability $1-\exp(-b_2 N)$,
\begin{align}
\frac{1}{K^d m^*} \sum_{x\in\L}\s(x) \le n^+ (1+\epsCG) + \frac{n^0}{m^*} - n^- (1-\epsCG) + \left(\frac{N}{K}\right)^d \O(K^{-1}).
\end{align}
By the argument from part (I) we can assume that $n^0\le N^{d-1}/\sqrt{K}$. The number of $\epsCG$-good boxes in $\L$ is therefore
\begin{align}\label{fajskfh3}
n^++n^- = \left(\frac{N}{K}\right)^d  [b_2^d-b_1^d+\O(K^{-1})].
\end{align}
By \eqref{fajskfh1}-\eqref{fajskfh3},
\[
n^-\le \left(\frac{N}{K}\right)^d
[\epsStab b_2^d + \epsCG(b_2^d-b_1^d)/2 +\O(K^{-1}) ] .
\]
Taking $\epsStab=\a/2$ and $\epsCG$ sufficiently small, we see that \eqref{volume constraint3} holds with high probability. Taking $2\stabilityConstant< \min\{1,\energyStabilityConstant(\epsStab)\}$ we have completed part (II) of the proof of \eqref{spanningProfilesConstant claim}.

Now for part (III). Choose $\fv$ such that \eqref{volume constraint}-\eqref{volume constraint3} are satisfied but \eqref{volume constraint2} is not. Let $1\le k< l< m< n\le 4R$ and let $\psi\in\sF(\fv)$. See Figure~\ref{fig:partIII}.

\begin{figure}
\begin{center}
\begin{picture}(0,0)%
\includegraphics{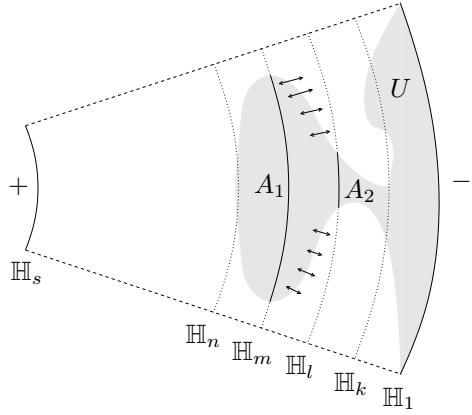}%
\end{picture}%
\setlength{\unitlength}{1326sp}%
\begingroup\makeatletter\ifx\SetFigFont\undefined%
\gdef\SetFigFont#1#2#3#4#5{%
  \reset@font\fontsize{#1}{#2pt}%
  \fontfamily{#3}\fontseries{#4}\fontshape{#5}%
  \selectfont}%
\fi\endgroup%
\begin{picture}(8220,7587)(661,-2506)
\put(6706,-2176){\makebox(0,0)[lb]{\smash{{\SetFigFont{10}{12.0}{\familydefault}{\mddefault}{\updefault}$\HH_k$}}}}
\put(5806,-1861){\makebox(0,0)[lb]{\smash{{\SetFigFont{10}{12.0}{\familydefault}{\mddefault}{\updefault}$\HH_l$}}}}
\put(4816,-1591){\makebox(0,0)[lb]{\smash{{\SetFigFont{10}{12.0}{\familydefault}{\mddefault}{\updefault}$\HH_m$}}}}
\put(3961,-1276){\makebox(0,0)[lb]{\smash{{\SetFigFont{10}{12.0}{\familydefault}{\mddefault}{\updefault}$\HH_n$}}}}
\put(7606,-2446){\makebox(0,0)[lb]{\smash{{\SetFigFont{10}{12.0}{\familydefault}{\mddefault}{\updefault}$\HH_1$}}}}
\put(8866,1604){\makebox(0,0)[lb]{\smash{{\SetFigFont{10}{12.0}{\familydefault}{\mddefault}{\updefault}$-$}}}}
\put(721,-106){\makebox(0,0)[lb]{\smash{{\SetFigFont{10}{12.0}{\familydefault}{\mddefault}{\updefault}$\HH_s$}}}}
\put(676,1514){\makebox(0,0)[lb]{\smash{{\SetFigFont{10}{12.0}{\familydefault}{\mddefault}{\updefault}$+$}}}}
\put(5221,1514){\makebox(0,0)[lb]{\smash{{\SetFigFont{10}{12.0}{\familydefault}{\mddefault}{\updefault}$A_1$}}}}
\put(6886,1424){\makebox(0,0)[lb]{\smash{{\SetFigFont{10}{12.0}{\familydefault}{\mddefault}{\updefault}$A_2$}}}}
\put(7741,3314){\makebox(0,0)[lb]{\smash{{\SetFigFont{10}{12.0}{\familydefault}{\mddefault}{\updefault}$U$}}}}
\end{picture}%
\end{center}
\caption{Part (III): The shaded region $U$ corresponds to the blocks with phase label $-1$.
The arrows indicate paths from $-1$ phase labels to $+1$ phase labels.
The black lines marking the intersection of the shaded region with $\HH_m$ and $\HH_l$ correspond to $A_1$ and $A_2$ in Lemma~\ref{area difference lemma}, respectively. The intersection of $\pd U$ with $\HH_{l,m}$ corresponds to the set $S$.
\label{fig:partIII}}
\end{figure}

We can apply Lemma~\ref{truncation} with $\Gamma$ equal to the set of mesoscopic boxes in $\HH_{k+1,n-1}$ in the neighborhood of a $0$ or a $-1$ box,
\begin{gather}\label{GammaIII}
\Gamma=\bigcup_{i\in I}\BB_K(i),\\
I = \{i:\exists j,\, \BB_K(i),\BB_K(j)\subset\HH_{k+1,n-1}, \psi(j)<1 \text{ and } \|i-j\|_\oo\le 1\}.\nonumber
\end{gather}
Note that $\Gamma$ is not necessarily connected.

Let $\psi\in\sF(\fv)$. Recall that the quantities $f_{\pd I}^+,f_{\pd I}^0,f_{\pd I}^-,f_I^\lra$ in the statement of Lemma~\ref{truncation} count the number of $+1$, $0$ and $-1$ phase labels in $\pd I$, and the maximum flow from plus to minus labels in $\Gamma$.
To estimate the quantity $f_I^\lra$ we need a geometric lemma.
Recall the definition of $\pd$ \eqref{pdU}.
\begin{lemma}\label{area difference lemma}
Suppose that $0<a_1<a_2$ and $U\subset\sA_\th$.
Let $A_i=U\cap \pd\sW_\th(a_i)$, $i=1,2$.
Let $S$ denote the portion of $\pd U$ contained in the closure of $\sW_\th(a_1,a_2)$.
Then \[\sH^{d-1}(S) \ge d^{-1/2}[\sH^{d-1}(A_1)-\sH^{d-1}(A_2)].\]
\end{lemma}
\begin{proof}
Let $P$ denote the linear projection $x\to (a_1/a_2)x$. $S$ must separate $A_1\sm P A_2$ from $(P^{-1}A_1) \sm A_2$. The surface area of $A_1\sm P A_2$ is at least $\sH^{d-1}(A_1)-\sH^{d-1}(A_2)$ as $P$ is a contraction.
Let $\bx\in\pd\sW_\th$.
Recall that $\sW_\th=\sA_\th\cap\sW_{2\pi}$; by the symmetry of $\sW_{2\pi}$, the angle between the vector $\bx$ and the normal vector $\bn^{\sW_\th}(\bx)$ is at most $\cos^{-1}(d^{-1/2})$.
\end{proof}
\noindent Lemma~\ref{area difference lemma} implies that there is a constant $c\in(0,1)$ such that
\[
f_I^\lra \ge c f_m^{-1} - c^{-1}(f_l^{-1}+f_l^0).
\]
Given $\hat{\a}$, if $\a$ is sufficiently small then we can choose $k< l< m< n$ such that
\begin{align}
f_k^0+f_k^{-1}&\le (b_2N/K)^{d-1}\times \hat{\a}\truncationConstantB c /(6\truncationConstant),\nonumber\\
f_l^0+f_l^{-1}&\le (b_2N/K)^{d-1}\times \hat{\a} c^2/6,\nonumber\\
f_m^{-1}&\ge (b_2N/K)^{d-1}\times \hat{\a},\label{bottle-neck}\\
f_n^0+f_n^{-1}&\le (b_2N/K)^{d-1}\times \hat{\a}\truncationConstantB c /(6\truncationConstant).\nonumber
\end{align}
Lemma~\ref{truncation} gives
\begin{align}\label{trunc_claim-}
\varphi^{J,(+,-),h}_\L(\mathrm{Int}_\psi \cap \mathrm{Ext}_\psi) &\le\exp\left(\truncationConstant (f_k^0+f_k^{-1}+f_n^0+f_n^{-1})K^{d-1} -\truncationConstantB f_I^\lra K^{d-1}\right)\nonumber\\
                                                                 &\le \exp(-\hat{\a}\truncationConstantB c(b_2N)^{d-1}/2).
\end{align}
We complete part (III) by checking that the right hand side of \eqref{trunc_claim-}, when multiplied by the size of the set $\sF(\fv)$ [cf. \eqref{size of sFfv} and \eqref{volume constraint}] is less than $\exp(-2\stabilityConstant b_2 N)$.

Before we start part (IV) of the proof of \eqref{spanningProfilesConstant claim}, we will give an isoperimetric inequality for $\pd\sW_\th$. The surface $\pd\sW_\th$ is $d-1$ dimensional, so subsets of $\pd\sW_\th$ have $d-2$ dimensional boundaries.
\begin{lemma}
\label{isoperimetric}
There is a positive constant $u=u(\th)$ such that for any $A\subset \pd\sW_\th$,
\[
\frac{\sL^{d-1}[A]}{\sL^{d-1}[\pd\sW_\th]}\le \frac12 \implies \sH^{d-2}[\pd A]  \ge u \left(\sL^{d-1}[A]\right)^{(d-2)/(d-1)}.
\]
\end{lemma}
\begin{proof}
First consider the case $\th=2\pi$.
Let $w$ denote the minimum $L_\oo$-distance between $\pd \sW_\th$ and the origin \eqref{w}.
By convexity, $\pd\sW_\th$ lies inside the $L_2$-annulus with inner radius $w$ and outer radius $wd$.

Consider the projection $P$ of $\pd\sW_\th$ onto the unit sphere $\sS^{d-1}$.
Associated with $P$ are two Radon-Nikodym derivatives, one for the $d-1$ dimensional Lebesgue measures on the domain and codomain, and one for the $d-2$ dimensional Hausdorff measures on the domain and codomain.
By symmetry and convexity (see \cite[Theorem~2.2.4]{Schneider}) both Radon-Nikodym derivatives are bounded away from $0$ and $\oo$. The result follows from L\'evy's isoperimetric inequality for the unit sphere.

A similar argument works when $0<\th<\pi$. Consider a projection from $\sW_\th$ to the $(d-1)$-dimensional unit ball.
\end{proof}
\noindent
Now for part (IV) of the proof of \eqref{spanningProfilesConstant claim}.
Let $\fv$ denote a spanning profile and let $\psi\in\sF(\fv)$.
Choose $k$ and $n$ such that $R\le k \le 2R \le n \le 3R$.
We will apply Lemma~\ref{truncation-} with $\Gamma$ defined according to \eqref{GammaIII}.

If \eqref{volume constraint2} holds with $\hat{\a}$ sufficiently small then Lemma~\ref{isoperimetric} implies that for some constant $c>0$,
\[
f_I^\lra \ge c \sum_{l=k+1}^{n-1} (f_l^{-1})^{(d-2)/(d-1)}.
\]
Let
\[
g_l:= \truncationConstantB f_l^0 K+ \truncationConstantB c (f_l^{-1})^{(d-2)/(d-1)} K^{d-1}.
\]
Lemma~\ref{truncation-} gives
\begin{align}\label{trunc_claim}
&\varphi^{J,(+,-),h}_\L(\mathrm{Int}_\psi\cap \mathrm{Ext}_\psi)\\
&\le\exp\left(\truncationConstant (f_k^0+f_k^{-1}+f_n^0+f_n^{-1})K^{d-1} + \b h | \Gamma| 4^d -\sum_{l=k+1}^{n-1}g_l\right).\nonumber
\end{align}
The term corresponding to the magnetic field in \eqref{trunc_claim} is bounded by \eqref{volume constraint2}. If $\hat\a$ is sufficiently small,
\[
\b h |\Gamma| 4^d \le \b h (12K)^d \sum_{l=k+1}^{n-1} f_l^0 + f_l^{-1} \le \frac{1}{4} \sum_{l=k+1}^{n-1} g_l.
\]
With reference to \eqref{size of sFfv} and the assumption that $h$ is small, the number of ways of choosing $(\psi(i):\BB_K(i)\subset \HH_{k,n})$ is at most
\[
\sum_{l=k}^n \exp(f_l^0 \O(\log N))\le  \exp\left((f_k^0+f_n^0)\O(\log N) + \frac14 \sum_{l=k+1}^{n-1} g_l\right).
\]
Thus if $h$ is sufficiently small,
\begin{align}\label{trunc_claim2}
\varphi^{J,(+,-),h}_\L(\Psi\in\sF(\fv)) \le \exp\left(2\truncationConstant (f_k^0+f_k^{-1}+f_n^0+f_n^{-1})K^{d-1}-\frac12 \sum_{l=k+1}^{n-1}g_l\right).
\end{align}
In our notation, \cite[(4.40)]{Stability-of-interfaces} states that if $\fv$ satisfies \eqref{volume constraint}-\eqref{volume constraint3} with $\a$ sufficiently small then there is a positive constant $c$ such that
\begin{align}
\label{4.40}
&\min_{R<k<2R} \left\{2\truncationConstant (f_k^0+f_k^{-1}) K^{d-1} - \frac12 \sum_{l=k+1}^{2R} g_l \right\} \le -c b_2 N, \text{ and}\\
&\min_{2R<n<3R} \left\{2\truncationConstant (f_n^0+f_n^{-1}) K^{d-1} - \frac12 \sum_{l=2R+1}^{n-1} g_l\right\} \le -c b_2 N.\nonumber
\end{align}
Part (IV) of the proof of inequality \eqref{spanningProfilesConstant claim} follows from \eqref{trunc_claim2} and \eqref{4.40}. This completes the proof of Proposition~\ref{decay-bc}.
\end{proof}

We will give three analogous results below. They follow, mutatis mutandis, from the proof of Proposition~\ref{decay-bc}.
Consider first the context of Proposition~\ref{lem:stabilityR}.

\begin{proposition}
\label{decay-bc-R}
Recall the event $\sC$ defined by \eqref{mu hat}.
Suppose that $b_1=0$ and $b_2(1-\epsHS)> \Bcritical+\epsC$.
Given $\epsC$ and $\epsHS$, with uniformly high $\QQ_\th$-probability
\begin{align*}
&\varphi^{J,-,h}_\L (\pd^-\L\lra\WW_\th(b_1,b_2(1-\epsHS))\mid\sC)\le \exp (-\stabilityConstant b_2/h).
\end{align*}
\end{proposition}
\begin{proof}
In part (II) of the proof of Proposition~\ref{decay-bc} replace Proposition~\ref{lem:stability} with Proposition~\ref{lem:stabilityR}.
\end{proof}

In the context of Proposition~\ref{lem:stability3}, the plus phase is dominant and so the minus boundary does not affect the bulk of the domain.

\begin{proposition}
\label{decay-bc3}
There is a constant $\stabilityConstant'=\stabilityConstant'(\epsHS,\Bmax)>0$ such that with uniformly high $\QQ_\th$-probability
\begin{align*}
\phi^{J,(-,+),h}_\L(\pd^-\L\lra \WW_\th(b_1+\epsHS,b_2)) \le \exp (-\stabilityConstant'/h).
\end{align*}
\end{proposition}

\begin{proof}
Proposition~\ref{decay-bc3} differs from Proposition~\ref{decay-bc} in that it shows that the inner (rather than the outer) boundary condition has limited influence. Let
\begin{align*}
&\HH_l=\WW_\th\left(b_1+\frac{2(4R-l)K}{wN},\ b_1+\frac{2(4R+1-l)K}{wN}\right),\q l=1,\dots,4R,\\
&\HH_l=\WW_\th\left(b_1+\frac{2(l-1)K}{wN},\ b_1+\frac{2 l K}{wN}\right),\q l=4R+1,\dots,S,
\end{align*}
with $R=\lfloor \epsHS N w/(8 K)\rfloor$ and $S=\lfloor (b_2-b_1)Nw/(2K)\rfloor$. The proof of Proposition~\ref{decay-bc}, mutandis mutandis, shows that a surface of $+1$ boxes separates $\HH_1$ from $\HH_{4R}$. In \eqref{volume constraint3}-\eqref{volume constraint2}, replace $b_2N/K$ with $N/K$. Proposition~\ref{lem:stability3} replaces Proposition~\ref{lem:stability} in part (II) of the proof.

Notice that the proof of part (III) has become slightly more flexible; the additional flexibility will be important below in the proof of Proposition~\ref{decay-bc2}. The quantity $\stabilityConstant'$ is allowed to depend on $\Bmax$. This means that Lemma~\ref{truncation-} can be used in place of Lemma~\ref{truncation}; the extra term due to the magnetic field can be controlled by taking $\a$, and therefore $|\Gamma|$, sufficiently small.

\end{proof}

Consider the context of Proposition~\ref{lem:stability2}. The minus phase is dominant so the plus boundary does not affect the bulk of the domain.
\begin{proposition}
\label{decay-bc2}
Suppose that $\mathsf{E}^\th(b_1+\epsHS)<\mathsf{E}^\th(b_2)$. There is a constant $\stabilityConstant''=\stabilityConstant''(\epsHS)>0$ such that with uniformly high $\QQ_\th$-probability
\begin{align*}
\phi^{J,(+,-),h}_\L(\pd^+\L\lra \WW_\th(b_1+\epsHS,b_2)) \le \exp (-\stabilityConstant''/h).
\end{align*}
\end{proposition}
\begin{proof}
The proof can be obtained from the proof of Proposition~\ref{decay-bc3} by swapping the roles of plus and minus. The sign of the magnetic field has to stay the same, but, for example, in \eqref{volume constraint3} replace $f_l^{-1}$ with $f_l^{+1}$, etc. Proposition~\ref{lem:stability2} replaces Proposition~\ref{lem:stability3} in part (II) of the proof.
\end{proof}

\subsection{Hausdorff stability implies spatial mixing}
\label{sec:mixing}
In Proposition~\ref{decay-bc2} we showed that under $\QQ_\th [\mu^{J,(+,-),h}_\L]$, with high probability, a surface of $-1$ mesoscopic blocks separates the region $\WW_\th'(b_1+\epsHS,b_2)$ from the plus boundary $\pd^+\L$. In two dimensions, by planar duality, this implies that there are no Ising spin-clusters connecting $\pd^+\L$ to $\WW_\th'(b_1+\epsHS,b_2)$. By the Ising model's domain Markov property, and monotonicity, we can compare $\QQ_\th [\mu^{J,(+,-),h}_\L]$ to $\QQ_\th [\mu^{J,(-,-),h}_\L]$ in $\WW_\th'(b_1+\epsHS,b_2)$.

In contrast in higher dimensions, especially when close to the critical temperature, the Ising spin-cluster associated with $\pd^-\L$ under $\mu^{J,(+,-),h}_\L$ may be much larger than the cluster associated with $\pd^-\L$ under the random-cluster representation $\phi^{J,(+,-),h}_\L$. We cannot make such a comparison. Instead we will appeal to a spatial mixing property, stated below as Proposition~\ref{mixing prop}.

Much is known about the spatial mixing properties of the Ising model in the absence of a magnetic field.
The difficulty here is that the magnetic field is acting to weaken the dominant phase.

We conjecture that Proposition~\ref{mixing prop}, and the coarse graining property, holds for all $\b>\bc$. 
If that is the case then Theorem~\ref{theorem:main} holds up to the critical point.
For simplicity we will reuse the constants $\stabilityConstant,\stabilityConstant'$ and $\stabilityConstant''$, adjusting their values if necessary.

\begin{proposition}\label{mixing prop}

There is a finite $\b_0$ such that if $\b>\b_0$ then for $\epsHS>0$, with uniformly high $\QQ_\th$-probability:
\begin{romlist}
\item
For $x\in\WW_\th(b_1,b_2(1-2\epsHS))$,
\begin{align*}
\varphi^{J,(+,-),h}_\L&\left(\s(x)=1 \,\big|\, \pd^-\L\nlra \WW_\th(b_1,b_2(1-\epsHS))\right) \\
\ge \varphi^{J,(+,+),h}_\L&(\s(x)=1) - \exp(-\stabilityConstant b_2 /h).
\end{align*}
\item
For $x\in\WW_\th'(b_1+2\epsHS,b_2)$,
\begin{align*}
\varphi^{J,(-,+),h}_\L&\left(\s(x)=1 \,\big|\, \pd^-\L\nlra \WW_\th(b_1+\epsHS,b_2)\right) \\
\ge \varphi^{J,(+,+),h}_\L&(\s(x)=1) - \exp(-\stabilityConstant'/h).
\end{align*}
\item
If $\mathsf{E}^\th(b_1+2\epsHS)<\mathsf{E}^\th(b_2)$, then for $x\in\WW_\th'(b_1+2\epsHS,b_2)$,
\begin{align*}
\varphi^{J,(+,-),h}_\L&\left(\s(x)=1 \,\big|\, \pd^+\L\nlra \WW_\th(b_1+\epsHS,b_2)\right) \\
\le \varphi^{J,(-,-),h}_\L&(\s(x)=1) + \exp(-\stabilityConstant''/h).
\end{align*}
\item
If $b_1>\Bcritical$, for $x\in\WW_\th'(b_1+\epsHS b_2,b_2)$,
\begin{align*}
&\varphi^{J,(+,-),h}_\L(\s(x)=1) \\
\le 
&\varphi^{J,-,h}_{\WW_\th'(b_2)}\left(\s(x)=1 \,\middle|\, \int_{\WW_\th(b_1)}\MM_K^- \dd\sL^d\ge (\Bcritical)^d\right)
+ \exp(-\stabilityConstant b_2 /h).
\end{align*}
\end{romlist}

\end{proposition}

\noindent The statement of Proposition~\ref{mixing prop} is fine tuned to suit our needs---we have only considered spatial mixing in Wulff-shaped regions. Also, the restriction in part~(iii) is stricter than necessary.

We will prove Proposition~\ref{mixing prop} after first showing how it can be used with the results of Section~\ref{sec:hausdorff}.
By part (i), in the context of Proposition~\ref{decay-bc} with uniformly high $\QQ_\th$-probability, for $x\in\WW_\th(b_1,b_2(1-2\epsHS))$,
\begin{align}\label{example:stability}
|\mu^{J,(+,-),h}_\L (\s(x)=1)-\mu^{J,+,h}_\L (\s(x)=1)| \le 2\exp (-\stabilityConstant b_2 /h).
\end{align}
By part (ii), in the context of Proposition~\ref{decay-bc3} with uniformly high $\QQ_\th$-probability, for $x\in\WW_\th'(b_1+2\epsHS,b_2)$,
\begin{align}\label{example:stability3}
|\mu^{J,(-,+),h}_\L (\s(x)=1)-\mu^{J,+,h}_\L (\s(x)=1)| \le 2\exp (-\stabilityConstant'/h).
\end{align}
By part (iii), in the context of Proposition~\ref{decay-bc2} with uniformly high $\QQ_\th$-probability, for $x\in\WW_\th'(b_1+2\epsHS,b_2)$,
\begin{align}\label{example:stability2}
|\mu^{J,(+,-),h}_\L (\s(x)=1)-\mu^{J,-,h}_\L (\s(x)=1)| \le 2\exp (-\stabilityConstant'' /h).
\end{align}
Suppose that $b_2>b_1>B_\mathrm{root}^\th$ and consider $x\in\WW_\th'(b_1+\epsHS b_2,b_2)$.
By part (iv), and by Proposition~\ref{lem:stability} applied to $\WW_\th(b_1)$ with $\epsStab=1-\Bcritical/b_1$, with uniformly high $\QQ_\th$-probability,
\begin{align}\label{example:stability-}
&|\mu^{J,(+,-),h}_{\WW_\th'(b_1,b_2)} (\s(x)=1)-\mu^{J,-,h}_{\WW_\th'(0,b_2)} (\s(x)=1)| \\
&\le \exp(-\energyStabilityConstant b_1/h^{d-1})+\exp (-\stabilityConstant b_2 /h).\nonumber
\end{align}
Proposition~\ref{mixing prop} is also relevant to the conditioned measure defined in \eqref{mu hat}.
With $b_1=0$, the total variation distance between $\hm_\L$ and $\mu^{J,-,h}_\L(\bcdot\mid \pd^-\L\nlra \WW_\th(b_2(1-\epsHS)))$ is bounded by Proposition~\ref{lem:stability4}, monotonicity, and Proposition \ref{decay-bc-R}.
Thus by part (i) of Proposition~\ref{mixing prop}, for some constant $c>0$, with uniformly high $\QQ_\th$-probability for $x\in\WW_\th'(b_2(1-2\epsHS))$,
\begin{align}\label{example:stability-R}
|\hm_\L (\s(x)=1)-\mu^{J,+,h}_\L (\s(x)=1)| \le \exp (-c/h).
\end{align}
If $b_2>b_1\ge\Bcritical+2\epsC$ then part (iv) can be used to compare $\mu^{J,(+,-),h}_\L$ and $\hm_{\WW_\th'(0,b_2)}$. For $x\in\WW_\th'(b_1+\epsHS b_2,b_2)$, with uniformly high $\QQ_\th$-probability for some positive constant $c>0$,
\begin{align}\label{example:stability-R-}
|\mu^{J,(+,-),h}_\L (\s(x)=1)-\hm_{\WW_\th'(0,b_2)} (\s(x)=1)| \le \exp (-c/h).
\end{align}

\begin{proof}[Proof of Proposition~\ref{mixing prop}]
We will prove part (i); the other parts are similar.
By monotonicity it is sufficient to show that for $x\in\WW_\th(b_1,b_2(1-2\epsHS))$,
\begin{align*}
&\mu^{J,(+,\f),h}_{\WW_\th(b_1,b_2(1-\epsHS))}\left(\s(x)=1\right)\ge \mu^{J,(+,+),h}_{\WW_\th(b_1,b_2(1-\epsHS))}\left(\s(x)=1\right) -\exp(-\stabilityConstant b_2 /h).
\end{align*}
Let $A$ denote the event that the inner- and outer-boundaries of $\WW_\th(b_2(1-2\epsHS),b_2(1-\epsHS))$ are separated by a set of plus spins blocking all paths between the two. Of course, the set of plus spins only needs to block paths composed entirely of edges with $J(e)=1$.
By monotonicity
\begin{align*}
\mu^{J,(+,\f),h}_{\WW_\th(b_1,b_2(1-\epsHS))}(\s(x)=1\mid A)\ge \mu^{J,(+,+),h}_{\WW_\th(b_1,b_2(1-\epsHS))}(\s(x)=1)
\end{align*}
so we need to show that
\begin{align*}
\mu^{J,(+,\f),h}_{\WW_\th(b_1,b_2(1-\epsHS))}(A)\ge 1-\exp(-\stabilityConstant b_2/h).
\end{align*}
We will do this using a stronger coarse-graining property.
\begin{definition}
Consider a box $\BB_K(i)\subset\L$. If
\begin{romlist}
\item $\BB_K(i)$ is $\epsCG$-good and
\item the $\s$-spin clusters composed of vertices with spin $-\s(\BB^\dagger_K(i))$ intersecting $\BB_K(i)$ have diameter at most $K/2$
\end{romlist}
then say that $\BB_K(i)$ is $\epsCG${\em-Ising-good}.
\end{definition}

With $p\in(\pc,1)$ fixed, as $\beta\to\oo$ the annealed random-cluster measure $\QQ[\phi^{J,h}]$ converges weakly to product measure with density $p$; the density of edges with $J(e)=1$ but $\om(e)=0$ goes to zero. Taking $K_0$ large, and then taking $\b$ large, we can make the $\QQ[\varphi^{J,+,0}_\L]$-probability that $\BB_{K_0}(i)\subset\L$ is $\epsCG$-Ising good arbitrarily close to $1$. By a standard renormalization argument we can find $\b_0,K_1$ and $c>0$ such that if $\b>\b_0$ and $K>K_1$ then $\BB_K(i)\subset\L$ is $\epsCG$-Ising-good with $\QQ[\varphi^{J,\f,0}_\L]$-probability $1-\exp(-c K)$.

The result now follows by adapting the proof of Proposition~\ref{decay-bc}. Substitute `$\epsCG$-Ising-good' for `$\epsCG$-good' in the definition of the phase labels and, because of the free outer boundary conditions, use Proposition~\ref{lem:stability4} in place of Proposition~\ref{lem:stability}.
If the profile $\fv$ associated with the phase label configuration $\Psi$ is not spanning then the event $A$ holds.

Parts (ii) and (iii) of Proposition~\ref{mixing prop} follow from the proofs of Propositions~\ref{decay-bc3} and \ref{decay-bc2}, respectively, by substituting `$\epsCG$-Ising-good' for `$\epsCG$-good'. Part (iv) follows from the proof of Proposition~\ref{decay-bc-R}; with high probability there is a surface of plus spins separating the inner- and outer-boundaries of $\WW_\th(b_1,b_1+\epsHS b_2)$ under $\mu^{J,-,h}_\L(\bcdot \mid \int_{\WW_\th(b_1)}\MM_K^- \dd\sL^d\ge (\Bcritical)^d)$.
\end{proof}

\subsection{Spectral gap of the dynamics}\label{sec:mix}

The Glauber dynamics for $\mu^{J,\zeta,h}_\L$ can be studied by introducing a block dynamics.
With $\epsBlock>0$, let $n=\lfloor (b_2-b_1)/\epsBlock\rfloor-1$. Consider a sequence of overlapping annuli that cover $\L$,
\begin{align*}
&\De_j=\WW_\th(b_1+(j-1)\epsBlock,b_1+(j+1)\epsBlock),\qq j=1,2,\dots,n-1,\\
&\De_n=\WW_\th'(b_1+(n-1)\epsBlock,b_2).
\end{align*}
Consider a block dynamics for $\mu^{J,\zeta,h}_\L$ with blocks $\De_1,\dots,\De_n$; update each block $\De_j$ at rate 1, resampling the block conditional on the configuration restricted to $\L\sm\De_j$.
\begin{lemma}\label{block-dynamics}
For $\epsSG>0$, if $\epsBlock$ and $h_0$ are sufficiently small and $0<h<h_0$,
\[
\mathrm{gap}(\L,\zeta,h) \ge \exp(-\epsSG/h^{d-1}) \mathrm{gap}(\L,\{\De_1,\dots,\De_n\},\zeta,h).
\]
\end{lemma}
\begin{proof}
$\sW_\th$ is a subset of $\sW_{2\pi}$, so we can assume $\th=2\pi$ without loss of generality.

Let $y^j_1,\dots,j^j_{|\De_j|}$ denote an ordering of the vertices in $\De_j$ such that the angle between $y^j_i$ and $\be_1$ is increasing with $i$.
For each vertex $y^j_i$ in block $\De_j$ consider the edge-boundary between $\{y^j_1,\dots,y^j_i\}$ and $\{y^j_{i+1},\dots,y^j_{|\De_j|}\}$; let $L$ denote the maximum (over $i=1,\dots,|\De_j|$ and $j=1,\dots,n$) size of the boundary. Given $\Bmax$, $Lh^{d-1}=\O(\epsBlock)$ as $\epsBlock\to 0$.

As noted in \cite{schonmann-shlosman}, the proof of \cite[Theorem 2.1]{martinelli} implies that for some $C,c>0$,
\begin{align*}
\mathrm{gap}(\L,\zeta,h)\ge \frac{c\exp(-C L)}{|\L|} \mathrm{gap}(\L,\{\De_1,\dots,\De_n\},\zeta,h).
\end{align*}
Choose $\epsBlock$ so that $CL< \epsSG/h^{d-1}$.
\end{proof}

We are now in a position to extend \cite[Propositions~3.5.1--3.5.3]{schonmann-shlosman} from the Ising model on $\ZZ^2$ to the dilute Ising model on $\ZZ^d$ with $d\ge 2$.

\begin{proposition}
\label{sg1}
Let $b_1=0$ and $\epsSG>0$. With uniformly high $\QQ_\th$-probability,
\begin{align*}
\mathrm{gap}(\L,+,h)\ge \exp(-\epsSG/h^{d-1}).
\end{align*}
\end{proposition}
\begin{proof}
Let $(\s_t)_{t\ge 0}$ denote a copy of the block dynamics Markov chain.
The graphical construction can be extended to the block dynamics by coupling from the past: if block $\De_j$ is to be updated at time $t$, use a copy of the regular graphical construction in $\De_j$ over the time interval $(-\oo,0]$ with boundary conditions $\s_{t-}$ to produce the new configuration $\s_t$. By monotonicity, $\s_t$ is an increasing function of the initial configuration $\s_0$.

When $t$ is sufficiently large, $\s_t$ is independent of $\s_0$; $\s_t$ then corresponds to a sample from the equilibrium distribution $\mu^{J,+,h}_\L$. Let $(\s^\mathrm{eqm}_t)_{t\ge0}$ denote a copy of the block dynamics Markov chain started in equilibrium.

We will show that with uniformly high $\QQ_\th$-probability, the probability that $\s_1=\s^\mathrm{eqm}_1$ is bounded away from zero. This implies that the spectral gap of the block dynamics is bounded away from zero and so the result follows by Lemma~\ref{block-dynamics}.

Say that an update of block $\De_j$ at time $t$ is {\em good} if the update maps all configurations that agree with $\s^\mathrm{eqm}_{t-}$ on $\L\sm\cup_{i=1}^j\De_i$ to configurations that agree with $\s^\mathrm{eqm}_t$ on the larger set $\L\sm\cup_{i=1}^{j-1}\De_i=\WW_\th'(b_1+j\epsBlock,b_2)$.
By monotonicity, if $\s_{t-}$ agrees with $\s^\mathrm{eqm}_{t-}\sim\mu^{J,+,h}_\L$ on $\L\sm\cup_{i=1}^j\De_i$ then
\begin{align}\label{s1s2}
\mu^{J,(-,+),h}_{\WW_\th'(b_1+(j-1)\epsBlock,b_2)} \lest \s_t \lest \mu^{J,(+,+),h}_{\WW_\th'(b_1+(j-1)\epsBlock,b_2)}.
\end{align}
Inequality \eqref{example:stability3} used with the sandwich \eqref{s1s2} gives a lower bound on the probability that the update is good; $\s_t$ and $\s^\mathrm{eqm}_t$ agree on $\WW_\th'(b_1+j\epsBlock,b_2)$ with probability at least $1-2|\L|\exp(-\stabilityConstant'/h)$.

With probability $\exp(-n)/n!$ there is an uninterrupted sequence of updates on $\De_n,\De_{n-1},\dots,\De_1$ in the time interval $[0,1]$. If all the updates are good, which occurs with probability at least $1-2n|\L|\exp(-\stabilityConstant' N)$, then $\s_1=\s^\mathrm{eqm}_1$. Note that $n\le \Bmax/\epsBlock$ so the probability of seeing such a sequence of updates is bounded away from zero uniformly over $b_2\in[0,\Bmax]$.
\end{proof}

\begin{proposition}
\label{sg3}
Let $b_1=0$ and $\epsSG>0$. One can choose $b_2$ slightly larger than $\Bcritical$ such that with high $\QQ_\th$-probability,
\begin{align*}
\mathrm{gap}(\L,-,h)\ge \exp(-\epsSG/h^{d-1}).
\end{align*}
\end{proposition}
\begin{proof}
Take $\epsBlock$ according to Lemma~\ref{block-dynamics}. Taking $n=\lfloor \Bcritical/\epsBlock\rfloor-1$, choose $b_2\in(\Bcritical,(n+2)\epsBlock)$ such that $\mathsf{E}^\th((n-1)\epsBlock)<\mathsf{E}^\th(b_2)$.
We can then follow the proof of Proposition~\ref{sg1}. The boundary conditions in \eqref{s1s2} should be changed to $(-,-)$ on the left-side and $(+,-)$ on the right-side. The probability of a block update being good is then bounded below using inequality \eqref{example:stability2} in place of \eqref{example:stability3}
\end{proof}
\begin{proposition}
\label{sg2}
Consider the case $b_1>\Bcritical$. Let $\epsSG>0$. With uniformly high $\QQ_\th$-probability,
\begin{align*}
\mathrm{gap} (\L,(+,-),h) \ge \exp( - \epsSG/h^{d-1} ).
\end{align*}
\end{proposition}
\begin{proof}
Say that an update of block $\De_j$ at time $t$ is {\em good} if the update maps all configurations that agree with $\s^\mathrm{eqm}_{t-}$ on $\WW_\th(b_1+(j-1)\epsBlock)=\L\sm\cup_{i=j}^n\De_i$ to configurations that agree with $\s^\mathrm{eqm}_t$ on $\L\sm\cup_{i=j+1}^n\De_i$.
The boundary conditions in \eqref{s1s2} should be changed to $(+,-)$ on the left-side and $(+,+)$ on the right-side. Inequality \eqref{example:stability} shows that updates are good with high probability.
If there is an uninterrupted sequence of good updates in the order $\De_1,\dots,\De_n$ in the time interval $[0,1]$ then $\s_1=\s^\mathrm{eqm}_1$.
\end{proof}

\section{Space-time cones and rescaling}\label{sec:big-proof}

In this section we will turn the heuristic description of plus-cluster nucleation from Section~\ref{sec:heuristic} into a proof of Theorem~\ref{theorem:main}. We will apply the results in Section~\ref{sec:technical} with two values of $\th$.

We will take $\th\in(0,\pi)$ to denote the argument of $\l_2^\th$ in the statement of the theorem.
We will consider regions with the shape $\WW_\th(b)$ for $b\in[\Bmin,\Bmax]$. The lower bound $\Bmin$ will be chosen to maximize the rate of nucleation of plus clusters. We will take $\Bmax$ to be the minimum value such that a translation of $\WW_{2\pi}(1.01 \Bcriticall)$ fits inside $\WW_\th(\Bmax)$; see parts 1 and 2 of Figure~\ref{fig:grow}.
By Proposition~\ref{catalyst size}, $\sW_{2\pi}(\Bcriticall)$ has diameter of order $1$ as $\b\to\oo$ and $\th\to 0$; $\sW_\th(\Bmax)$ must have volume of order $\th^{-1}$. The Wulff shape $\sW_\th(b)$ has volume $b^d$ so $\Bmax$ must be of order $\th^{-1/d}$. By \eqref{Q-th-def-} the probability of the event conditioned on in the definition of $\QQ_\th$ is
\begin{align}\label{prob cat}
\exp(-\cdil \th^{-1}/h^{d-1}) \qtext{with} \cdil=\O\left(\log\frac{1}{1-p}\right).
\end{align}
This gives the density of nucleation sites in $\ZZ^d$.
We will show that at these nucleation sites, droplets of plus phase form at the rate $\exp(-\mathsf{E}^\th_\c/h^{d-1})$.

We must then show that the clusters of plus-phase can spread out from the sheltered nucleation sites. We do this by considering the full Wulff shape $\sW_{2\pi}$.
In areas of typical dilution, sufficiently large Wulff-shaped droplets of plus phase expand with high probability.
With reference to \eqref{example:stability} we will take $\Bmaxx$ large so that $\mu^{J,-,h}_{\WW_{2\pi}(\Bmaxx)}$ provides a good approximation to the equilibrium measure $\mu^{J,h}$ in a neighborhood of the origin.

\begin{figure}
\begin{center}
\begin{picture}(0,0)%
\includegraphics{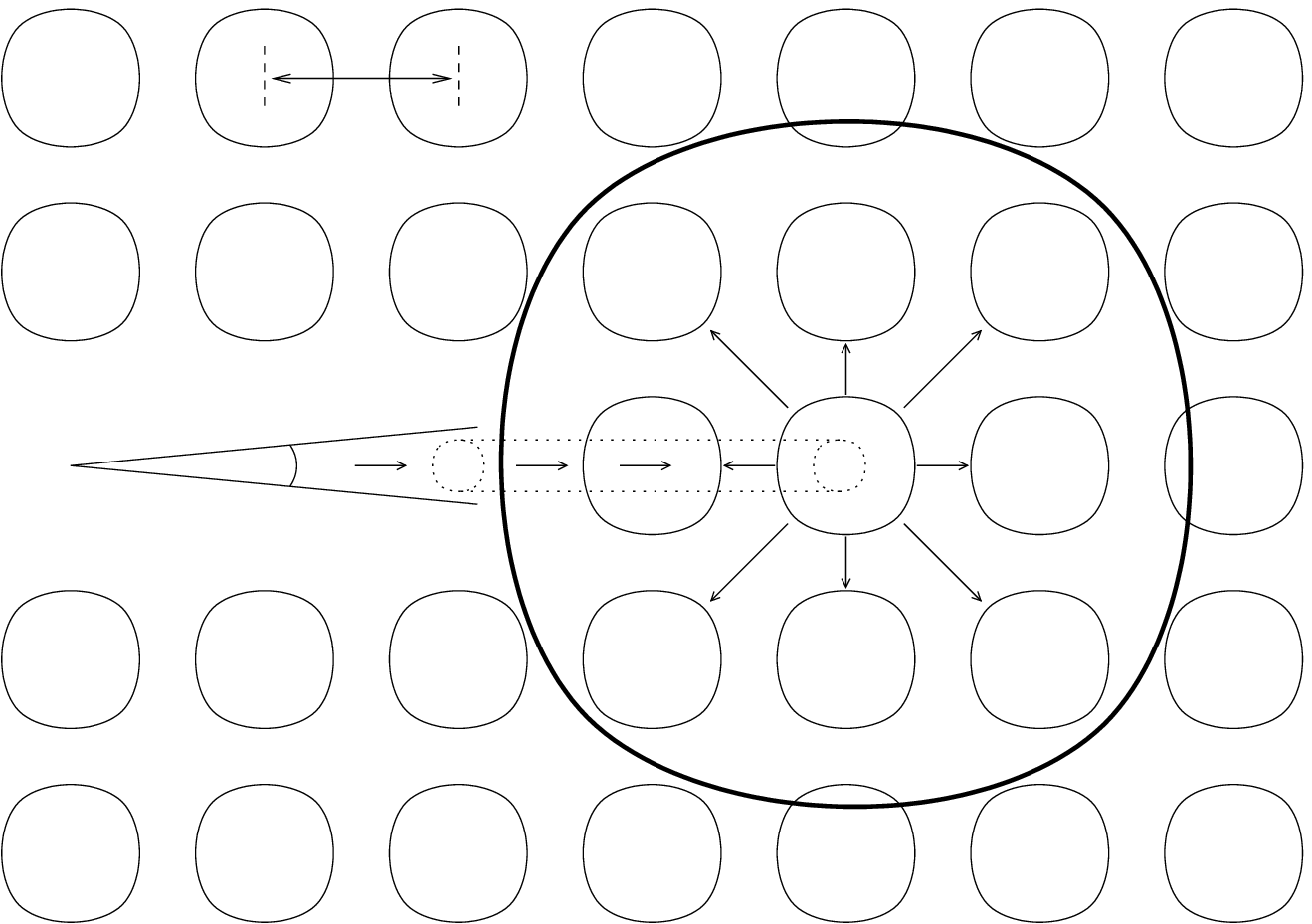}%
\end{picture}%
\setlength{\unitlength}{2735sp}%
\begingroup\makeatletter\ifx\SetFigFont\undefined%
\gdef\SetFigFont#1#2#3#4#5{%
  \reset@font\fontsize{#1}{#2pt}%
  \fontfamily{#3}\fontseries{#4}\fontshape{#5}%
  \selectfont}%
\fi\endgroup%
\begin{picture}(9084,6384)(-3191,-2353)
\put(-44,749){\makebox(0,0)[lb]{\smash{{\SetFigFont{11}{13.2}{\familydefault}{\mddefault}{\updefault}$2$}}}}
\put(2611,749){\makebox(0,0)[lb]{\smash{{\SetFigFont{11}{13.2}{\familydefault}{\mddefault}{\updefault}$3$}}}}
\put(2836,974){\makebox(0,0)[lb]{\smash{{\SetFigFont{11}{13.2}{\familydefault}{\mddefault}{\updefault}$4$}}}}
\put(3376,2774){\makebox(0,0)[lb]{\smash{{\SetFigFont{11}{13.2}{\familydefault}{\mddefault}{\updefault}$5$}}}}
\put(-809,3584){\makebox(0,0)[lb]{\smash{{\SetFigFont{11}{13.2}{\familydefault}{\mddefault}{\updefault}$lN$}}}}
\put(-1304,749){\makebox(0,0)[lb]{\smash{{\SetFigFont{11}{13.2}{\familydefault}{\mddefault}{\updefault}$1$}}}}
\end{picture}%
\end{center}
\caption{\label{fig:grow}Illustration of a $\QQ$-catalyst. The nucleation event occurs at rate $\exp(-\mathsf{E}^\th_\c/h^{d-1})$ and consists of the following steps. (1) A droplet of plus phase with the shape $\WW_\th(\Bmin)$ forms in a region of high dilution that resembles $\WW_\th(\Bmax)$ [Proposition~\ref{droplet creation}]. (2) The droplet expands in the sheltered region  to cover a copy of $\WW_{2\pi}(1.01 \Bcriticall)$ [Proposition~\ref{prop322}]. (3) The droplet of plus phase spreads to the right [Proposition~\ref{prop322'}] and (4) expands to cover $\WW_{2\pi}(\Bminn)$ [Proposition~\ref{prop322} with $\th$ taken to be $2\pi$].
There is now a droplet of plus phase at the center of a $\QQ$-conductive site of the rescaled lattice. The droplet expands (5) to cover $\WW_{2\pi}(\Bmaxx)$ [Proposition~\ref{prop321}] which contains neighboring $\QQ$-conductive sites.
}
\end{figure}

\subsection{The graphical construction in space-time regions}

Before we give the proof of Theorem~\ref{theorem:main}, we need to extend the Ising dynamics to allow the size of the graph to change with time.
With $\Gamma_0,\Gamma_1,\dots,\Gamma_n \subset \ZZ^d$ and $t_0<t_1<\dots<t_{n+1}$, consider the space-time region
\begin{align}\label{ST}
\Gamma=\mathrm{ST}(\Gamma_0,\dots,\Gamma_n;t_0<\dots<t_{n+1}):=\bigcup_{i=0}^n \Gamma_i \times [ t_i , t_{i+1} ].
\end{align}
The graphical construction for the Ising model $\mu^{J,\zeta,h}_\L$ described in Section~\ref{sec:heat bath} can be extended to $\Gamma$.
\begin{romlist}
\item Let $s$ denote the start time.
\item Let $\xi$ denote an initial configuration compatible with boundary conditions $\zeta$ at time $s$, i.e. if $s\in[t_i,t_{i+1})$ then $\xi\in\Si_{\Gamma_i}^\zeta$.
\item Let $\s^{s,\xi}_{\Gamma,\zeta,h;s}=\xi$.
\item If a vertex $x$ is added to the dynamics at time $t_i$ (i.e. $x\in\Gamma_i\sm\Gamma_{i-1}$) then the spin $\s^{s,\xi}_{\Gamma,\zeta,h;t_i}(x)$ is taken to be $\xi(x)$ to match the boundary conditions. The spin at $x$ may then change with each arrival of the corresponding Poisson process.
\item If $x$ is removed from the dynamics at time $t_i$ (i.e. $x\in \Gamma_{i-1}\sm\Gamma_i$) then the spin at $x$ is immediately switched to $\xi(x)$ to conform to the boundary conditions.
\end{romlist}
The graphical construction of $\PP_J$ allows us to link together the Ising dynamics run in overlapping space-time regions.
This can be used to chain together the different steps involved in the growth of a region of plus-phase.
\begin{remark}\label{concat}
Consider two space-time regions such that the top layer of the first region covers the start of the second region:
\begin{align*}
\Gamma= \mathrm{ST}&(\Gamma_0,\Gamma_1,\dots,\Gamma_m;t_0<t_1<\dots<t_{m+1}),\\
\De=  \mathrm{ST}&(\De_0,\De_1,\dots,\De_n;u_0<u_1<\dots<u_{n+1}),\\
\Gamma_m=\De_0 &\text{ and } u_0= t_m<t_{m+1}\le u_1.
\end{align*}
If $\s^{t_0,\xi}_{\Gamma,-,h;t_{m+1}}= \s^{t_m,+}_{\Gamma,-,h;t_{m+1}}$ and $\s^{u_0,+}_{\De,-,h;u_{n+1}}=\s^{u_n,+}_{\De,-,h;u_{n+1}}$ then
\begin{align*}
\s^{t_0,\xi}_{\Gamma\cup\De,-,h;u_{n+1}}= \s^{u_n,+}_{\De,-,h;u_{n+1}}.
\end{align*}
\end{remark}

\subsection{Droplet creation in a Summertop cone}
Let $\th\in(0,\pi)$ and $\d>0$.
By Proposition~\ref{sg3} we can choose $\Bmin\in(\Bcritical,B_\mathrm{root}^\th)$ such that with high $\QQ_\th$-probability,
\begin{align}\label{gap asd}
\mathrm{gap}(\WW_\th(\Bmin),-,h)\ge \exp(-\d/(2h^{d-1})).
\end{align}
Let $\L=\WW_\th(\Bmin)$. Heuristically, we expect critical droplets to form in $\L$ at rate $\exp(-\mathsf{E}^\th_\c/h^{d-1})$.
Let $\hm_\L=\mu^{J,-,h}_\L ( \bcdot \mid \sC)$ denote the conditional measure defined by \eqref{mu hat} with $\epsC=(\Bmin-\Bcritical)/3$.

\begin{proposition}
\label{droplet creation}
With high $\QQ_\th$-probability we can construct a random variable $\hat\s\sim\hm_\L$ such that the event $\{\hat\s=\s^{0,-}_{\L,-,h;\exp(\d/h^{d-1})}\}$ has probability $\exp(-\mathsf{E}^\th_\c/h^{d-1})$ and is independent of the value of $\hat\s$.
\end{proposition}
\begin{proof}
Taking $a=0$ in the last inequality in the proof of Proposition~\ref{lem:stabilityR}, we can assume that $\mu^{J,-,h}_\L(\sC)\ge 2\exp(-\mathsf{E}^\th_\c/h^{d-1})$.
By \eqref{gap asd} and a Markov chain mixing inequality (i.e. \cite[(59)]{SchonmannSlowDropletRelaxation}) the total variation distance between $\s^{0,-}_{\L,-,h;\exp(\d/h^{d-1})}$ and $\mu^{J,-,h}_\L$ is less than $\exp(-\mathsf{E}^\th_\c/h^{d-1})$.
\end{proof}

\subsection{Growing in a Summertop cone}
In this section we will use the ``inverted space-time pyramids'' of \cite{schonmann-shlosman} to show that under $\QQ_\th$, droplets of plus phase tends to expand from $\WW_\th(\Bmin)$ to $\WW_\th(\Bmax)$ with high probability.

With $\d>0$, and with reference to \eqref{increasing sets} and \eqref{ST}, consider the space-time region
\begin{gather*}
\triangledown=\triangledown(\Bmin,\Bmax,\d,\th):=\mathrm{ST}(\L_0,\dots,\L_n;t_0<\dots<t_{n+1}),\\
\L_i:=\De_\th^{i+|\WW_\th(\Bmin)|},\q t_i:=i\exp(\d/h^{d-1}), \q n:=|\WW_\th(\Bmax)|-|\WW_\th(\Bmin)|.
\end{gather*}
For each $i$ let $\hm_{\L_i}=\mu^{J,-,h}_{\L_i} ( \bcdot \mid \sC)$ with $\epsC=(\Bmin-\Bcritical)/3$.

For $\n\in\Si_{\WW_\th(\Bmin)}^-$ consider the event
\begin{align*}
G_\n:=\left\{\s^{0,\n}_{\triangledown,-,h;t_{n+1}}=\s^{t_n,+}_{\triangledown,-,h;t_{n+1}}\right\}.
\end{align*}
For $\n\in\sC$, $G_\n$ describes the plus phase spreading from $\WW_\th(\Bmin)$ to $\WW_\th(\Bmax)$ in time $t_n$.
The event $G_\n$ depends only on the elements of the graphical construction contained in $\triangledown$.
Here is an extension of \cite[Proposition 3.2.2]{schonmann-shlosman} to the dilute Ising model.
\begin{proposition}\label{prop322}
There are positive constants $\d_0,C,\propconstanta$ such that if $0<\d\le\d_0$ then with high $\QQ_\th$-probability
\begin{align*}
\int  \PP_J(G_\n) \dd\hm_{\L_0}(\n) \ge 1-C \exp(-\propconstanta/h).
\end{align*}
\end{proposition}

\begin{proof}[Proof of Proposition~\ref{prop322}]
We will assume that $J$ belongs to a certain event with high $\QQ_\th$-probability; the set is defined implicitly by our use of results from Section~\ref{sec:technical}.

For $i=1,\dots,n$ and $\zeta\in\Si_{\L_{i-1}}^-$ let $G^i_\zeta=\left\{\s^{t_{i-1},\zeta}_{\triangledown,-,h;t_{i+1}}=\s^{t_i,+}_{\triangledown,-,h;t_{i+1}}\right\}$;
\begin{align*}
G^1_\n \cap \left( \bigcap_{i=2}^n G^i_+\right) \implies G_\n.
\end{align*}
By monotonicity $G^i_\zeta\subset G^i_+$. It is sufficient to show that for each $i$,
\begin{align}\label{collect0}
\int  \PP_J(G^i_\zeta) \dd\hm_{\L_{i-1}}(\zeta) \ge 1-C \exp(-\propconstanta/h).
\end{align}
With reference to Proposition~\ref{lem:stabilityR}, if we start the dynamics with initial distribution $\hm_{\L_{i-1}}$ we expect to stay inside $\sC$ for a long time.
Let $(\hat{\s}^{s,\zeta}_{\triangledown,-,h;t})_{t\ge s}$ denote the Markov chain obtained from the graphical construction by suppressing any jumps from $\sC$ to $\sC^\c$.
By introducing a stopping time
\[
\tau^\zeta_i=\inf\{t\ge t_i : \s^{t_i,\zeta}_{\triangledown,-,h;t}\not=\hat{\s}^{t_i,\zeta}_{\triangledown,-,h;t}\},
\]
we will see that the modified dynamics are likely to agree with the regular dynamics over the interval $[t_i,t_{i+1}]$.

Let $\s^x$ denote the configuration obtained from $\s$ by flipping the spin at $x$, and let
\begin{align*}
\pd\sC=\{\s\in\sC:\exists x, \s^x \in \sC^\c\}.
\end{align*}
Proposition~\ref{lem:stabilityR} gives an upper bound on $\hm_{\L_i}(\pd\sC)$. Given $\Bmin$ we can find $\d_0>0$ such that
\begin{align}
\label{phi_bound}
\hm_{\L_i}(\pd\sC) \le \exp(-3\d_0/h^{d-1}).
\end{align}
If the starting state $\zeta$ is sampled from $\hm_{\L_i}$ then the process $(\hat{\s}^{t_i,\zeta}_{\triangledown,-,h;t})_{t\in[t_i,t_{i+1}]}$ is stationary.
By \eqref{phi_bound} (cf. \cite[(2.12)]{schonmann-shlosman}), if $\d\le\d_0$ and $h$ is sufficiently small,
\begin{align}\label{stopping time}
\int \PP_J(\tau^\zeta_i\le t_{i+1}) \dd\hm_{\L_i}(\zeta) \le \exp(-\d_0/h^{d-1}).
\end{align}
For some $x$, $\L_i=\L_{i-1}\cup\{x\}$. We will need a bound on the effect of adding this extra vertex has on the conditional Ising measures $(\hm_{\L_i})_i$.
By the Ising model's finite-energy property, for any $h_0>0$,
\begin{align}\label{def alpha}
\a:=\inf_{0<h<h_0} \inf_{\zeta\in\Si} \inf_J \inf_{s=\pm1} \mu^{J,\zeta,h}_{\{0\}} (\s(0)=s) >0.
\end{align}
By the Ising model's Markov property, for $\zeta\in\sC\cap\Si_{\L_{i-1}}^-$,
\begin{align*}
\hm_{\L_i} (\zeta)/\hm_{\L_{i-1}} (\zeta) =
\mu^{J,-,h}_{\L_i}(\s(x)=-1\mid\sC)\ge \a.
\end{align*}
Therefore (cf. \cite[(3.28)]{schonmann-shlosman}),
\begin{align}\label{3.28}
\int & \PP_J((G^i_\zeta)^\c)\dd\hm_{\L_{i-1}}(\zeta)\\
=&\int  \PP_J(\s^{t_{i-1},\zeta}_{\triangledown,-,h;t_{i+1}} \not= \s^{t_i,+}_{\triangledown,-,h;t_{i+1}})\dd\hm_{\L_{i-1}}(\zeta)\nonumber\\
\le &\int \PP_J(\s^{t_i,\zeta}_{\triangledown,-,h;t_{i+1}} \not= \s^{t_i,+}_{\triangledown,-,h;t_{i+1}}) \dd\hm_{\L_{i-1}}(\zeta)
+ \int \PP_J(\tau^\zeta_{i-1}\le t_i) \dd\hm_{\L_{i-1}}(\zeta) \nonumber\\
\le \a^{-1} &\int \PP_J(\s^{t_i,\zeta}_{\triangledown,-,h;t_{i+1}}\not=\s^{t_i,+}_{\triangledown,-,h;t_{i+1}}) \dd\hm_{\L_i}(\zeta) + \exp(-\d_0/h^{d-1}).\nonumber
\end{align}
Set
\begin{align}\label{b alpha}
b_\gamma=(1-\gamma)\Bcritical+\gamma \Bmin, \qq \gamma\in[0,1].
\end{align}
Thus $\Bcritical=b_0<b_{1/3}<b_{2/3}<b_1=\Bmin$.
Choose $b$ such that $\WW_\th'(b)=\L_i$.
By monotonicity and the invariance of the modified dynamics with respect to $\hm_{\WW_\th'(b)}$,
\begin{align*}
    &\int  \PP_J(\s^{t_i,\zeta}_{\triangledown,-,h;t_{i+1}}\not=\s^{t_i,+}_{\triangledown,-,h;t_{i+1}}) \dd\hm_{\WW_\th'(b)}(\zeta)\\
\le &\int \PP_J(\s^{t_i,+}_{\triangledown,-,h;t_{i+1}}>\hat{\s}^{t_i,\zeta}_{\triangledown,-,h;t_{i+1}})+\hm_{\WW_\th'(b)} (\tau_i^\zeta\le t_{i+1})  \dd\hm_{\WW_\th'(b)}(\zeta)\\
\le &\int \sum_{y\in\WW_\th'(b)} \Big\{ \PP_J(\s^{t_i,+}_{\triangledown,-,h;t_{i+1}}(y)=1)-\PP_J(\hat{\s}^{t_i,\zeta}_{\triangledown,-,h;t_{i+1}}(y)=1)\Big\} \dd\hm_{\WW_\th'(b)}(\zeta)\\
&+\exp(-\d_0/h^{d-1}) \\
\le &\sum_{y\in\WW_\th(b_{2/3})} \Big\{\PP_J(\s^{0,+}_{\WW_\th'(b),+,h;\exp(\d/h^{d-1})}(y)=1)-\hm_{\WW_\th'(b)}(\s(y)=1)\Big\}\\
+&   \sum_{y\in\WW_\th'(b_{2/3},b)} \Big\{\PP_J(\s^{0,+}_{\WW_\th'(b_{1/3},b),(+,-),h;\exp(\d/h^{d-1})}(y)=1)-\hm_{\WW_\th'(b)}(\s(y)=1)\Big\}\\
&+\exp(-\d_0/h^{d-1}).
\end{align*}
For $y\in\WW_\th(b_{2/3})$, by Proposition~\ref{sg1} and Markov chain mixing \cite[(59)]{SchonmannSlowDropletRelaxation},
\begin{align*}
&|\PP_J(\s^{0,+}_{\WW_\th'(b),+,h;\exp(\d/h^{d-1})}(y)=1)-\mu^{J,+,h}_{\WW_\th'(b)}(\s(y)=1)|\\
&\le \exp\left[-\exp(\d/h^{d-1})\mathrm{gap}(\WW_\th'(b),+,h)\right]\Big/\mu^{J,+,h}_{\WW_\th'(b)}(\s= +)\\
&\le \exp\left[-\exp(\d/(2h^{d-1}))\right].
\end{align*}
Similarly for $y\in\WW_\th'(b_{2/3},b)$, by Proposition~\ref{sg2},
\begin{align*}
&|\PP_J(\s^{0,+}_{\WW_\th'(b_{1/3},b),(+,-),h;\exp(\d/h^{d-1})}(y)=1)-\mu^{J,(+,-),h}_{\WW_\th'(b_{1/3},b)}(\s(y)=1)|\\
&\le \exp\left[-\exp(\d/h^{d-1})\mathrm{gap}(\WW_\th'(b_{1/3},b),(+,-),h)\right] \Big/ \mu^{J,+,h}_{\WW_\th'(b_{1/3},b)}(\s=+)\\
&\le \exp\left[-\exp(\d/(2h^{d-1}))\right].
\end{align*}
By the above
\begin{align}
&\int \nonumber \PP_J(\s^{t_i,\zeta}_{\triangledown,-,h;t_{i+1}}\not=\s^{t_i,+}_{\triangledown,-,h;t_{i+1}}) \dd\hm_{\WW_\th'(b)}(\zeta)\\
& \le\sum_{y\in\WW_\th(b_{2/3})}\mu^{J,+,h}_{\WW_\th'(b)}(\s(y)=+1)-\hm_{\WW_\th'(b)}(\s(y)=+1)\label{collect1}\\
&+   \sum_{y\in\WW_\th'(b_{2/3},b)}\mu^{J,(+,-),h}_{\WW_\th'(b_{1/3},b)}(\s(y)=+1)-\hm_{\WW_\th'(b)}(\s(y)=+1)\nonumber\\
&+   \exp(-\d_0/h^{d-1})+ |\WW_\th'(b)|\exp\left[-\exp(\d/(2h^{d-1}))\right].\nonumber
\end{align}
By \eqref{example:stability-R} and \eqref{example:stability-R-},
\begin{align}\label{collect2}
   & \sum_{y\in\WW_\th(b_{2/3})}\mu^{J,+,h}_{\WW_\th'(b)}(\s(y)=+1)-\hm_{\WW_\th'(b)}(\s(y)=+1)\\
  +& \sum_{y\in\WW_\th'(b_{2/3},b)}\mu^{J,(+,-),h}_{\WW_\th'(b_{1/3},b)}(\s(y)=+1)-\hm_{\WW_\th'(b)}(\s(y)=+1)\nonumber\\
\le& |\WW_\th'(b)| \exp(-c/h).   \nonumber
\end{align}
Inequality \eqref{collect0} now follows by \eqref{3.28}, \eqref{collect1} and \eqref{collect2}.
\end{proof}

\subsection{Escaping from Summertop-cones}

In Proposition~\ref{prop322} we considered space-time pyramids.
Consider now ``space-time parallelepipeds''. From now on we will write $2\pi$ in place of $\th$ to make it clear that $\th$ refers to the angle of the catalyst cone.
Let $a$ denote a positive constant and let $b=1.01 \Bcriticall$. We can find a sequence of graphs $\L_0,\L_1,\dots,\L_n$ such that
\begin{romlist}
\item $\L_0=\WW_{2\pi}(b)$,
\item $\L_n=\WW_{2\pi}(b)+aN\be_1$,
\item $\L_{i+1}$ differs from $\L_i$ by adding a vertex or removing a vertex,
\item for any $i$, for some $k_i\in(0,aN)$, $\L_i$ differs from $\WW_\th(b)+k_i\be_1$ by at most a mesoscopic layer of vertices around the boundary, and
\item $n=\O(a b^{d-1}/h^d)$,
\end{romlist}
Let
\begin{align}\label{lozenge}
\lozenge=\lozenge(a,b,\d)=\mathrm{ST}(\L_0,\dots,\L_n;t_0<\dots<t_{n+1}),
\end{align}
with $t_i:=i\exp(\d/h^{d-1})$. In Figure~\ref{fig:grow}, the dotted lines indicate the area swept out by a space-time parallelepiped that starts inside the copy of $\WW_\th(\Bmax)$.

For $\n\in\Si_{\L_0}^-$ consider the event
\begin{align*}
G_\n:=\left\{\s^{0,\n}_{\lozenge,-,h;t_{n+1}}=\s^{t_{n},+}_{\lozenge,-,h;t_{n+1}}\right\}.
\end{align*}
Here is an extension of Proposition~\ref{prop322} to $\lozenge$.
Let $\epsC=(b- \Bcriticall)/3$ and let $\hm_{\L_i}:=\mu^{J,-,h}_{\L_i} ( \bcdot \mid \sC_i)$ with
\[\sC_i=\left\{\s:\int_{\sW_\th(\Bcritical+\epsC)+(k_i/N)\be_1} \MM_K^{-}\dd \sL^d \ge (\Bcritical)^d\right\}.
\]
\begin{proposition}\label{prop322'}
Let $\lozenge$ be defined according to \eqref{lozenge} with $\d\le\d_0$.
There are positive constants $C,\propconstanta$ such that with high $\QQ$-probability
\begin{align*}
\int \PP_J(G_\n) \dd\hm_{\L_0}(\n) \ge 1-C \exp(-\propconstanta/h).
\end{align*}
\end{proposition}
\begin{proof}
We can adapt the proof of Proposition~\ref{prop322}, showing that \eqref{collect0} holds when, for example, $\L_i=\L_{i-1}\sm\{x\}$ for some vertex $x$ on the boundary of $\L_{i-1}$. Let $(\hat{\s}^{s,\zeta}_{\lozenge,-,h;t})_{t\ge s}$ denote the Markov chain obtained from the graphical construction by suppressing any jumps from $\sC_{i-1}$ to $\sC_{i-1}^\c$.
The only place where the change is important is in inequality \eqref{3.28}. Recall that the spin of vertices leaving $\lozenge$ are set to $-1$. Let $\mu$ denote the measure obtained by sampling from $\hm_{\L_{i-1}}$ and then setting the spin at $x$ equal to $-1$.
Let $\mu'=\mu^{J,-,h}_{\L_i}(\bcdot\mid\sC_{i-1})$.
By the definition of $\a$ \eqref{def alpha},
\[
\mu(\zeta) \le \a^{-1}\mu'(\zeta), \qq \zeta\in\Si_i^-.
\]
In place of \eqref{3.28} we have that
\begin{align*}
\int & \PP_J((G^i_\zeta)^\c)\dd\hm_{\L_{i-1}}(\zeta)\\
=&\int  \PP_J(\s^{t_{i-1},\zeta}_{\lozenge,-,h;t_{i+1}} \not= \s^{t_i,+}_{\lozenge,-,h;t_{i+1}})\dd\hm_{\L_{i-1}}(\zeta)\\
\le &\int \PP_J(\s^{t_i,\zeta}_{\lozenge,-,h;t_{i+1}} \not= \s^{t_i,+}_{\lozenge,-,h;t_{i+1}}) \dd\mu(\zeta)
+ \int \PP_J(\tau^\zeta_{i-1}\le t_i) \dd\hm_{\L_{i-1}}(\zeta) \\
\le \a^{-1}&\int \PP_J(\s^{t_i,\zeta}_{\lozenge,-,h;t_{i+1}}\not=\s^{t_i,+}_{\lozenge,-,h;t_{i+1}}) \dd\mu'(\zeta) + \exp(-\d_0/h^{d-1}).
\end{align*}
The rest of the proof follows mutatis mutandis.
\end{proof}

\subsection{Growth on a rescaled lattice}\label{sec:rescaling}
In Proposition~\ref{prop322} we require $\Bmin> \Bcritical$. If in addition $\Bmin>B_\mathrm{root}^\th$ then we get the following stronger result corresponding to \cite[Proposition 3.2.1]{schonmann-shlosman}.
\begin{proposition}\label{prop321}
Let $\Bmaxx>\Bminn>B_\mathrm{root}^{2\pi}$ and $\d>0$.
Consider $\triangledown=\triangledown(\Bminn,\Bmaxx,\d,2\pi)$.
There are positive constants $C,\propconstanta$ such that  with high $\QQ$-probability
\begin{align*}
\int  \PP_J(G_\n) \dd\mu^{J,-,h}_{\WW_{2\pi}(\Bminn)}(\n)\ge 1-C \exp(-\propconstanta/h).
\end{align*}
Moreover, $\propconstanta$ is a function of $\Bminn$ and $\propconstanta\to\oo$ as $\Bminn\to\oo$.
\end{proposition}

\begin{proof}[Proof of Proposition~\ref{prop321}]
Taking $\th=2\pi$, the proof of this proposition is very similar to the proof of Proposition~\ref{prop322}.
Define $b_\gamma=(1-\gamma)B_\mathrm{root}^{2\pi}+\gamma \Bminn$ in place of \eqref{b alpha}.
We can then simply replace $\hm_{\L_i}$ with $\mu^{J,-,h}_{\L_i}$.
The need for the modified dynamics and the stopping time has disappeared; the $\exp(-\d_0/h^{d-1})$ terms can be removed from the proof.

Let $\epsHS=(\Bminn-B_\mathrm{root}^{2\pi})/(8\Bminn)$. In \eqref{collect2}, \eqref{example:stability} and \eqref{example:stability-} replace \eqref{example:stability-R} and \eqref{example:stability-R-}, respectively. We can therefore replace the term $\exp(-c/h )$ with $2\exp(-\stabilityConstant\Bminn/h)$.
This yields the claim that $\propconstanta\to\oo$ as $\Bminn\to\oo$.
\end{proof}

\begin{proof}[Proof of Theorem~\ref{theorem:main}]
Let $\l$, $\l_2^\th$, $\thmconstant$ and $f$ refer to the corresponding quantities in the statement of the theorem. Let $\d=(\l-\l_2^\th)/3>0$.

The idea of a droplet of plus phase growing can be formalized using Proposition~\ref{prop321}.
With reference to Proposition~\ref{prop321}, choose $\Bminn$ such that $\propconstanta\ge \thmconstant+1$. Let $l=\text{diameter}(\sW_{2\pi}(\Bminn+1))$. With reference to \eqref{example:stability}, take $\Bmaxx$ to be greater than $3 (\Bminn+1)$ and large enough that for some constant $C$,
\begin{align*}
|\mu^{J,-,h}_{\WW_{2\pi}(\Bmaxx)}(f) -\mu^{J,h}(f)| \le C\|f\|_\oo \exp(-\thmconstant/h).
\end{align*}
Define a collection of overlapping translations of $\triangledown=\triangledown(\Bminn,\Bmaxx,\d,2\pi)$: let
\begin{align*}
\triangledown_{x,i}=\triangledown+ (lNx,iT) , \qq (x,i)\in\ZZ^d\times\NN,
\end{align*}
where $T$ denotes the time from the start of the first slice of $\triangledown$ to the start of the final slice of $\triangledown$,
\[
T=|\WW_{2\pi}(\Bminn,\Bmaxx)|\exp(\d/h^{d-1}).
\]
Time-wise, the top slice of $\triangledown_{x,i}$ overlaps the bottom slice of $\triangledown_{x,i+1}$.
We have chosen $\Bminn,l$ and $\Bmaxx$ so that
\begin{align*}
[\WW_{2\pi}(\Bminn) +  lN \be_1] &\cap \WW_{2\pi}(\Bminn) = \emptyset, \q\text{and}\\
[\WW_{2\pi}(\Bminn) +  lN \be_1] &\subset \WW_{2\pi}(\Bmaxx).
\end{align*}
If $\|x-y\|_1=1$, then $\triangledown_{x,0}$ and $\triangledown_{y,0}$ do not intersect at time 0, but they then `invade' each other: at time $T$, $\triangledown_{x,0}$ covers $\triangledown_{y,1}$.

Say that $x\in\ZZ^d$ is $\QQ$-{\em conductive} if, translated by $lNx$, the $\QQ$-event from Proposition~\ref{prop321} holds; $x$ is $\QQ$-conductive with high $\QQ$-probability. When $h$ is small the $\QQ$-conductive vertices form a supercritical site-percolation type of process on $\ZZ^d$. Let $G_{+,x,i}$ denote the translation by $(lNx,iT)$ of the $\PP_J$-event $G_+$ from Proposition~\ref{prop321}; space-time paths of $G_{+,x,i}$-events show how clusters of plus phase spread out once they have formed.

We will say $x\in \ZZ^d$ is a $\QQ$-catalyst if the event conditioned on in the definition of $\QQ_\th$, translated by $lNx$, occurs.
Let $D$ denote the density of $\QQ$-catalysts \eqref{prob cat}.

For a $\QQ$-catalyst to be effective, the edges that do not need to be closed should have typical dilution.
Choose $k$ minimal such that
\[
\WW_\th(\Bmax) \cap [\WW_{2\pi}(\Bmaxx)+klN\be_1] = \emptyset.
\]
With reference to Figure~\ref{fig:grow}, define a $\PP_J$-measurable event corresponding to the nucleation and escape of a plus droplet,
\[
\mathrm{Nuc}_x:=\left\{\s^{0,-}_{\triangledown_{(x+k\be_1,0)},-,h;T+\exp(\d/h^{d-1})}= \s^{T,+}_{\triangledown_{(x+k\be_1,0)},-,h;T+\exp(\d/h^{d-1})}\right\}.
\]
Figure~\ref{fig:grow} illustrates how $\mathrm{Nuc}_x$ can be written as the concatenation of the events described in Propositions~\ref{droplet creation}-\ref{prop321}; in the applications of Propositions \ref{droplet creation}-\ref{prop322'} take the value of $\d$ to be $\min\{\d_0,(\l-\l_2^\th)/6\}$.

We will say that a $\QQ$-catalyst $x$ is {\em good} if
\begin{romlist}
\item $x+k\be_1$ is $\QQ$-conductive, and
\item $\PP_J(\mathrm{Nuc}_x)\ge \exp\left(-\mathsf{E}^\th_\c/h^{d-1}\right)/2$.
\end{romlist}
$\QQ$-catalysts are good with high $\QQ$-probability. If $x$ is a good $\QQ$-catalyst, let $\mathrm{Nuc}_{x,i}$ denote $\mathrm{Nuc}_x$ translated $iT$ forward in time.

Let $\mathrm{Con}_M(x,y)$ denote the $\QQ$-event that $x$ and $y$ are joined by a simple path of exactly $M$ $\QQ$-conductive vertices, and let
\[
A:=\{x\in\ZZ^d: x-k\be_1 \text{ is a good }\QQ\text{-catalyst and }\mathrm{Con}_M(x,0)\},
\]
Take $M$ maximal such that $\triangledown_{0,3M}$ finishes before time $\exp(\l/h^{d-1})$.
By a Peierls argument, there is a constant $c>0$ such that with high $\QQ$-probability $\{|A|\ge c D M^d\}$. Assume that $|A|\ge c D M^d$.

The expected number of $(x,i)\in A\times\{0,1,\dots,M\}$ such that $\mathrm{Nuc}_{x,i}$ occurs is
\[
\exp\left(-\mathsf{E}^\th_\c/h^{d-1}\right)/2\cdot cDM^d\cdot (M+1) \ge \exp\left(\frac{\d(d+1)}{h^{d-1}}\right) \gg 1.
\]
We can assume that $\mathrm{Nuc}_{x,i}$ does occur for some $(x,i)\in A \times \{0,1,\dots,M\}$.
As $x\in A$, there is path $y_0,y_1,\dots,y_M\in\ZZ^d$ of $\QQ$-conductive sites from $y_0=x$ to $y_M=0$. The time between $\triangledown_{x,i}$ and $\triangledown_{0,3M}$ is between $2MT$ and $3MT$.

The growth of the region of plus-phase along the path $y_0,\dots,y_M$ corresponds to a directed percolation cluster on the graph $\{0,1,\dots,M\}\times\{i,i+1,\dots,3M\}$; see Figure~\ref{fig:m}.
By a Peierls argument, with high probability there is a space-time path
\[
((y_{j_k},k) : k=i,\dots,3M;\ j_i=0,\ j_{3M}=M \text{ and }\forall k, |j_k-j_{k+1}|\le 1)
\]
such that $G_{+,y_{j_k},k}$ occurs for $k=i,i+1,\dots,3M$.
\end{proof}

\begin{figure}
\begin{center}
\begin{picture}(0,0)%
\includegraphics{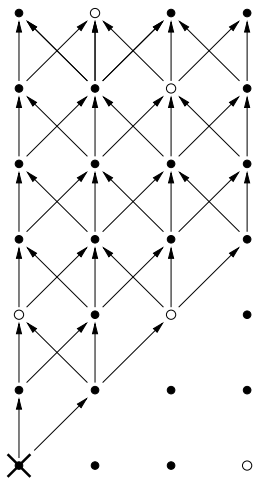}%
\end{picture}%
\setlength{\unitlength}{2072sp}%
\begingroup\makeatletter\ifx\SetFigFont\undefined%
\gdef\SetFigFont#1#2#3#4#5{%
  \reset@font\fontsize{#1}{#2pt}%
  \fontfamily{#3}\fontseries{#4}\fontshape{#5}%
  \selectfont}%
\fi\endgroup%
\begin{picture}(2931,5631)(-167,2471)
\end{picture}%
\end{center}
\caption{\label{fig:m}
The x marks a nucleation event $\mathrm{Nuc}_{x,i}$.
The horizontal axis corresponds to a path of length $M=3$ from $y_0=x\in A$ to the origin $y_M=0$. The vertical axis corresponds to time. 
The arrows indicate how the region of plus-phase can spread to neighboring points of the rescaled lattice [Proposition~\ref{prop321}].
The black dots indicate points of the rescaled space-time lattice where $G_{+,y,k}$ occurs.
The spread of the plus phase is thus bounded below by a supercritical directed-percolation cluster.}
\end{figure}

This completes the proof of Theorem~\ref{theorem:main} part (i). Part (ii) follows by Proposition~\ref{catalyst size}.

\section*{Acknowledgments}
TB and MW acknowledge the support of the French Ministry of Education through the ANR 2010 BLAN 0108 01 grant. 
BG thanks the Fondation Sciences Math\'ematiques de Paris for funding a postdoctoral fellowship at the \'Ecole Normale Sup\'erieure.

%%\bibliography{catalyst}
%%\bibliographystyle{abbrv}

\end{document}